\documentclass[11pt, a4paper]{amsart}

\usepackage{amsmath,graphicx, amsthm, amssymb, epsfig}
\usepackage{xcolor}
\usepackage[normalem]{ulem}
\numberwithin{equation}{section}

\usepackage{hyperref}
\hypersetup{
    colorlinks=true,
    linkcolor=blue,
    filecolor=magenta,      
    urlcolor=cyan,
    pdftitle={Overleaf Example},
    pdfpagemode=FullScreen,
}

\newcommand{\e}{\epsilon}

\newcommand{\de}{\delta}
\newcommand{\br}{\mathbb{R}}
\newcommand{\N}{\mathbb{N}}
\newcommand{\BZ}{\mathbb{Z}}
\newcommand{\ik}{\varphi}
\newcommand{\pa}{\partial}

\newcommand{\al}{\alpha}
\newcommand{\la}{\lambda}
\newcommand{\om}{\omega}
\newcommand{\coi}{C_0^{\infty}}
\newcommand{\ioi}{\int_0^{\infty}}
\newcommand{\be}{\begin{equation}}
\newcommand{\ee}{\end{equation}}
\newcommand{\on}{\upsilon}
\newcommand{\dd}{\text{d}}
\newcommand{\op}{\text{Op}}

\newcommand{\cha}{\check \alpha}

\newcommand{\chq}{\check q}

\newcommand{\cht}{\check t}
\newcommand{\chv}{\check v}
\newcommand{\chw}{\check w}
\newcommand{\chx}{\check x}
\newcommand{\chy}{\check y}

\newcommand{\hxi}{\hat  \xi}

\newcommand{\heta}{\hat  \eta}

\newcommand{\CG}{\mathcal G}
\newcommand{\BP}{\mathbb P}
\newcommand{\nrec}{N_\e^{\text{rec}}}
\newcommand{\frec}{f_\e^{\text{rec}}}
\newcommand{\Eb}{\mathbb E}

\def\bs#1\es{
    \begin{equation}\begin{split}
    #1
    \end{split}\end{equation}
}

\newcommand{\bma}{\begin{pmatrix}}
\newcommand{\ema}{\end{pmatrix}}

\newcommand{\us}{\mathcal U}
\newcommand{\vs}{\mathcal V}

\newcommand{\s}{\mathcal S}
\newcommand{\T}{\mathcal T}
\newcommand{\R}{\mathcal R}
\newcommand{\CF}{\mathcal F}

\newcommand{\I}{\mathcal I}

\newcommand{\CB}{\mathcal{B}}

\newcommand{\tsp}{\text{supp}\,}

\newcommand{\tb}{\tilde B}

\newcommand{\tsum}{\textstyle\sum}

\newtheorem{theorem}{Theorem}[section]
\newtheorem{lemma}[theorem]{Lemma}

\theoremstyle{definition}
\newtheorem{assumptions}[theorem]{Assumption}
\newtheorem{definition}[theorem]{Definition}

\begin{document}

\title[Noise in Iterative Reconstruction]{Local Characterization of Noise in Iterative Reconstruction of the Generalized Radon Transform}
\author[A Katsevich]{Alexander Katsevich$^1$}
\thanks{$^1$Department of Mathematics, University of Central Florida, Orlando, FL 32816 (Alexander.Katsevich@ucf.edu).}

\begin{abstract}
We study noise in iterative reconstruction from discrete noisy data of a generalized Radon transform in the plane. Our approach builds on Local Reconstruction Analysis (LRA), a framework for analyzing reconstructions at the native scale. We establish that the rescaled reconstruction error converges in distribution to a zero-mean Gaussian random field with explicitly computable covariance, providing a complete local characterization of noise in iterative reconstruction. Numerical experiments show strong agreement with the theoretical predictions. Combined with earlier deterministic results, our findings complete the analysis of iterative reconstruction at the native scale with respect to the two most fundamental limitations: the discreteness of the data and the presence of noise.
\end{abstract}


\maketitle

\section{Introduction}\label{sec:intro}

\subsection{Two strands of Local Reconstruction Analysis (LRA). Main results}
Let $\R$ denote the generalized Radon transform (GRT), which integrates a function $f$ defined on an open set $\us \subset \br^n$ (the image domain) over a family of surfaces parameterized by points in a set $\vs \subset \br^n$ (the data domain). The GRT data of $f$, denoted $\hat f(y) = (\R f)(y)$ for $y \in \vs$, are sampled on a rectangular grid $\{y_j\}_{j \in \mathbb Z^n}\subset\vs$ with step size proportional to a small parameter $\e > 0$. If the data are noisy, the measurements take the form
\be
g_j := (\R f)(y_j) + \on_j,\ j \in J := \{ j \in \BZ^n : y_j \in \vs \},
\ee
where $\{\on_j\}_{j\in J}$ are independent random variables.

We begin by considering reconstructions of the form $f_0 = \R^* \CB \hat f$, where $\R^*$ is the adjoint transform and $\CB$ is a pseudo-differential operator ($\Psi$DO). Let $\ik$ be a compactly supported, sufficiently smooth interpolation kernel, and let $g_\e$ denote the interpolant of the discrete data $\{g_j\}_{j\in J}$ (see \eqref{interp-data}). The reconstruction from discrete data is then $f_\e = \R^* \CB g_\e$. This is of filtered backprojection (FBP) type: interpolation and the application of $\CB$ constitute the filtering step, while $\R^*$ performs backprojection.

LRA, introduced by the author in \cite{Katsevich2017a}, provides a systematic framework for studying local properties of such reconstructions. The central idea is to derive an explicit asymptotic formula for $f_\e$ in an $\e$-neighborhood of a point $x_0 \in \us$. This local approximation serves as a surrogate reconstruction at the native $\sim\e$ scale: once available, it can be readily analyzed from any perspective -- edge smoothing, random-error covariance, aliasing, or feature detection -- without further algorithm-specific derivations.

Two main strands of LRA have been developed:
\begin{itemize}
\item In the \emph{deterministic setting}, the data are noise-free ($\on_j \equiv 0$), and one studies effects such as edge smoothing and aliasing artifacts in $f_\e$ \cite{Katsevich2021a, Katsevich2023a, Katsevich_2025_BV}.
\item In the \emph{stochastic setting}, by linearity, one may assume $f \equiv 0$ and analyze reconstructions from noise ($g_j = \on_j$). This yields a detailed description of the local reconstruction error caused solely by random noise \cite{AKW2024_1, katsevich2024a}.
\end{itemize}

Since exact inversion formulas exist for only a few GRTs, iterative reconstruction is often necessary. In a recent work by the author \cite{katsevich2025}, LRA was extended to iterative reconstruction for discrete GRT data in the plane. The reconstruction $f_\e$ is obtained as the solution of the quadratic minimization problem
\be\label{min pb 0}
f_\e = \text{argmin}_{f \in H_0^1(\us_b)} \Psi(f), \quad
\Psi(f) := \Vert \R f - g_\e \Vert_{L^2(\vs)}^2 + \kappa \e^{3} \Vert \pa_x f \Vert_{L^2(\us)}^2,
\ee
where $\kappa > 0$ is a fixed regularization parameter, $\us_b \Subset \us$ is a subdomain, and $H_0^1(\us_b)$ is the closure of $\coi(\us_b)$ in the Sobolev $H^1(\br^2)$ norm. The factor $\e^3$ in the Tikhonov term ensures that the reconstruction resolution remains at the native scale $\e$.

The analysis in \cite{katsevich2025} concerns the deterministic setting ($\on_j \equiv 0$). The main result is an accurate formula describing how sharp edges in $f$ are smoothed in $f_\e$. The present paper complements that work by addressing the stochastic setting. Since $\Psi$ is quadratic, the reconstruction depends linearly on the data. Thus it suffices to study the case $f \equiv 0$ and $g_j \equiv \on_j$. 
In this case the reconstructed image is denoted $\nrec(x)$. In the spirit of LRA, we focus on an $O(\e)$ neighborhood of a generic point $x_0 \in \us_b \subset \br^2$. 

Let $C(\bar D)$ denote the space of continuous functions on a bounded Lipschitz domain $D\subset \br^2$. The main result of this paper is that, under suitable assumptions on the first four moments of $\{\on_j\}_{j\in J}$ and on $x_0$, the following limit exists:
\be\label{noise rec 0}
N^{\text{rec}}(\chx; x_0) = \lim_{\e \to 0} N_\e^{\text{rec}}(x_0 + \e \chx), \quad \chx \in\bar D.
\ee
Here, $N^{\text{rec}}$ and $N_\e^{\text{rec}}$ are viewed as $C(\bar D)$-valued random variables (random fields), and the limit is understood in the sense of probability distributions. We also prove that $N^{\text{rec}}(\chx; x_0)$, $\chx\in \bar D$, is a zero-mean Gaussian random field (GRF) with explicitly computable covariance. Finally, we present numerical experiments demonstrating strong agreement between theoretical predictions and simulated reconstructions.

As in \cite{katsevich2025}, our emphasis is on the most fundamental limitation to image quality: discrete noisy data. Although numerical implementation requires finite-dimensional approximation, its effect can be made negligible compared to data discretization and noise, so our analysis isolates the intrinsic contribution of noise.

\subsection{Related Work} 
Classical analyses of noise in CT focus on analytic inversion schemes such as FBP. The book by Kak and Slaney \cite[Sec.~5.2]{kak2001a} derives explicit expressions for how additive measurement noise propagates through FBP, including the expectation and variance of the reconstructed image. A related treatment is given in \cite[Sec.~16.2]{eps08}. However, the resulting expressions are not available in closed form and are therefore difficult to use for detailed theoretical analysis. 

Kernel-type estimators for tomographic problems with additive noise were studied in \cite{Korostelev1991,Tsybakov_92}, where minimax rates were obtained both pointwise and in $L^2$. Related results on maximal deviations appear in \cite{Bissantz2014}, while \cite{cavalier1998,Cavalier2000} established pointwise asymptotic efficiency for density estimation from noise-free Radon data on random grids. A common feature of these works is the reliance on additional smoothing at a scale much larger than the sampling step, which entails loss of resolution and requires strong smoothness assumptions on $f$. None of them derive the distribution of the reconstruction error, whereas our analysis works at the native resolution and yields a complete local characterization of the error. 

Bayesian inversion has also been developed for tomographic settings \cite{MNP2019,siltanen2003,vanska2009}. For Gaussian white noise, \cite{MNP2019} showed asymptotic normality of the posterior and MAP estimator for smooth test functionals $\int f\psi$, i.e. reconstructions at macroscopic scales. By contrast, we quantify pointwise error uniformly over $O(\e)$ neighborhoods. 

Another perspective is semiclassical analysis \cite{stefanov2023}, which characterizes the empirical spatial mean and variance of reconstruction noise for a single experiment as $\e\to0$. Our approach differs by describing the entire limiting error distribution across repeated reconstructions. 

Applied studies (e.g. \cite{Noo_noise_2008,Divel2020}) have examined image-domain noise with numerical and semi-empirical methods. These works provide valuable heuristics but do not yield theoretical PDFs of the reconstructed noise in small neighborhoods. 

Finally, the existing work on stochastic LRA \cite{AKW2024_1, katsevich2024a} does characterize noise in the reconstructed image, but the analysis is confined to FBP-type reconstruction.

\subsection{Significance. Organization of the paper}
Empirical noise studies in tomography typically require many independent experiments: one generates noisy data, reconstructs, and then averages results to estimate statistics such as variance or correlation. For iterative methods, however, each reconstruction involves solving a large optimization problem, so repeating this process many times is computationally expensive.

The significance of this work is twofold. First, it eliminates the need for extensive empirical sampling. By extending LRA to the stochastic iterative setting, we obtain a rigorous description of the entire limiting distribution of the reconstruction error at the native scale. In effect, our analysis provides directly what empirical studies would reveal only after a prohibitive number of reconstructions. This advances the theoretical understanding of errors in iterative methods while also offering practical value: reliable predictions of reconstruction error without repeated costly simulations.

Second, this work broadens the scope of LRA itself. Until recently, LRA had been applied only to FBP-type algorithms, yet exact FBP inversion formulas exist for only a few transforms. By extending LRA to iterative reconstruction, we expand its applicability to the much wider class of GRTs for which only iterative methods are available. Taken together with the deterministic analysis in \cite{katsevich2025}, the present results complete the study of reconstruction at the native scale with respect to the two most fundamental phenomena: data discreteness and noise. In turn, this opens the way to applying additional analyses to iterative reconstruction, such as the statistical microlocal analysis recently developed for FBP reconstruction \cite{abhishek2025}.

The remainder of the paper is organized as follows. Section~\ref{sec:basic notation} introduces the basic notation. In Section~\ref{sec:setting} we describe the problem setup, state the assumptions, and present the iterative reconstruction algorithm. The main results, Theorems~\ref{thm:Lyapunov_nd} and~\ref{GRF_thm}, are given in Section~\ref{sec:main_res}. Properties of the reconstructed image are examined in Section~\ref{sec:reg_sol}. Subsection~\ref{ssec:CRT example} illustrates how our results apply to the classical Radon transform. The proofs of the main theorems are contained in Sections~\ref{sec:grid recon} and~\ref{sec:prf GRF}. Section~\ref{sec:numerics} reports numerical experiments that confirm our theoretical predictions. Finally, the proofs of auxiliary results are collected in Appendices~\ref{sec:eta props}--\ref{sec:tight}.

\section{Basic notation, function spaces, symbols, operators}\label{sec:basic notation}

Denote $\N=\{1,2,\dots\}$ and $\N_0=\{0\}\cup\N$. Let $U\subset\br^n$, $n\in\N$, be a domain, which is defined as a non-empty, connected, open set. For clarity, the zero vector in $\br^n$, $n\ge2$, is denoted $0_n$.

We identify tangent and cotangent spaces on Euclidean spaces with the underlying Euclidean spaces. Thus, for a function $f(x)$ defined on a domain $U\subset\br^n$ we have
\bs
&\pa_x f(x):=f_x^\prime(x):=(\pa_{x_1} f(x),\dots,\pa_{x_n}f(x))^T,\\ 
&\dd_x f=(\pa_{x_1} f(x),\dots,\pa_{x_n}f(x)),\\
&|\pa_x f(x)|=|\dd_x f(x)|=\big[(\pa_{x_1} f(x))^2+\dots+(\pa_{x_n}f(x))^2\big]^{1/2}.
\es
We use $f_x^\prime$ or $\pa_x f$ interchangeably depending on what is more convenient from the notational perspective. 


The Sobolev space $H^s(\br^n)$, $s\in\br$, is the space of all tempered distributions $f$ for which 
\be
\Vert f\Vert_{H^s(\br^n)}^2=(2\pi)^{-n}\int_{\br^n} |\tilde f(\xi)|^2(1+|\xi|^2)^s\dd\xi<\infty.
\ee
$H_0^s(U)$ is the closure of $\coi(U)$ in $H^s(\br^n)$.

\begin{definition}
Given a domain $U\subset\br^n$ and $r\in\br$, $S^r(U)$ denotes the vector-space of $C^\infty(U\times (\br^n\setminus0_n))$ functions, $\tilde B(x,\xi)$, having the following properties
\bs\label{symbols def}
&|\pa_x^m \tilde B(x,\xi)|\le c_m|\xi|^{-n+\de}\ \forall m\in\N_0^n,x\in U,0<|\xi|\le 1;\\
&|\pa_x^{m_1} \pa_\xi^{m_2}\tilde B(x,\xi)|\le c_{m_1,m_2}|\xi|^{r-|m_2|},\ \forall m_1\in\N_0^n,m_2\in\N_0^n,x\in U,|\xi|\ge 1;
\es
for some constants $\de>0$ and $c_m,c_{m_1,m_2}>0$. 
\end{definition}

The elements of $S^r$ are called symbols of order $r$. We modify the conventional definition slightly to allow for symbols to be non-smooth at the origin. This makes our analysis more streamlined, but otherwise has no effects. The space of the corresponding $\Psi$DOs, given by 
\be
(\op(\tilde B(x,\xi))f)(y):=\frac1{(2\pi)^n}\int_{\br^n}\tilde B(x,\xi)\tilde f(\xi)e^{-i\xi\cdot y}\dd\xi,\ 
f\in\coi(U),
\ee
is denoted $L^r(U)$. The notation $\us_b\Subset\us$ means that the closure $\overline{\us_b}$ is compact and $\overline{\us_b}\subset \us$.

%

For convenience, throughout the paper we use the following convention. If a constant $c$ is used in an equation, the qualifier ‘for some $c>0$’ is assumed. If several $c$ are used in a string of (in)equalities, then ‘for some’ applies to each of them, and the values of different $c$’s may all be different. 

Additional function spaces, notation, and conventions are introduced as needed.

\section{Setting of the problem, assumptions, and reconstruction algorithm}\label{sec:setting}

\subsection{GRT and its properties}\label{ssec:grt}
Consider a GRT, which integrates over the curves $\s_y:=\{x\in\us:p=\Phi(x,\al)\}$, $y:=(\al,p)\in \vs:=\I_\al\times\I_p$:
\bs\label{R def}
(\R f)(y)=& \int_{\us} f(x)W(x,y)\de(p-\Phi(x,\al))\dd x\\
=& \frac1{2\pi}\int_{\br} \int_{\us} f(x)W(x,y)e^{i\nu(p-\Phi(x,\al))} \dd x\dd\nu,\ y\in\vs,
\es
where $W\in C^\infty(\us\times\vs)$. An open set $\us\subset\br^2$ is the image domain, and $\vs=\I_\al\times \I_p$ is the data domain. Here and everywhere below we use the convention that whenever $y$, $\al$, and $p$ appear in the same equation or sentence, then $y=(\al,p)$.
\begin{assumptions}[Properties of the GRT -  I]\label{ass:Phi}$\hspace{1cm}$
\begin{enumerate}
\item\label{Ial} $\I_\al\subset\br$ is a closed interval with the endpoints identified, so it can be viewed as a circle,
\item $\I_p\subset\br$ is an open interval (or all of $\br$), and 
\item\label{Phi sm} $\Phi\in C^\infty(\us\times \I_\al)$ and $\Phi(\us\times \I_\al)\subset\I_p$. 
\end{enumerate}
\end{assumptions}
Both $\us$ and $\vs$ are endowed with the usual Euclidean metric. 

\begin{assumptions}[Properties of the GRT - II]\label{geom GRT}$\hspace{1cm}$
\begin{enumerate}
\item\label{li} $\dd_x\Phi(x,\al)$ and $\dd_x\Phi_\al^\prime(x,\al)$ are linearly independent on $\us\times \I_\al$.
\item\label{bolk} The Bolker condition: if $\Phi(x_1,\al)=\Phi(x_2,\al)$ for some $x_1,x_2\in\us$, $x_1\not=x_2$, and $\al\in \I_\al$, then $\Phi_\al^\prime(x_1,\al)\not=\Phi_\al^\prime(x_2,\al)$. 
\item\label{visib} For each $(x,\xi)\in T^*\us$, there exist $\al\in \I_\al$ and $\nu\in\br$ such that $\xi=\nu\dd_x\Phi(x,\al)$.
\item\label{pos W} $W\in C^\infty(\us\times\vs)$ and $W(x,y)>c$ on $\us\times\vs$.
\item\label{invert} For any domain $\us_b\Subset\us$ there exists $c=c(\us_b)$ such that
\be\label{inverse bound}
\Vert f\Vert_{H^{-1/2}(\br^2)}\le c \Vert \R f\Vert_{L^2(\vs)}\text{  for any  } f\in H_0^{-1/2}(\us_b).
\ee
\end{enumerate}
\end{assumptions}

Assumption~\ref{geom GRT}\eqref{li} implies that the curves $\s_y$ are smooth. See \cite{katsevich2025} for a more detailed discussion of these assumptions.

\subsection{GRT data and the reconstruction algorithm}\label{ssec:data_recon}

Realistic data can usually be represented as the sum of a useful signal (the GRT of some function, $\hat f=\R f$) and noise. Hence we assume that our data are $g(y_j)=\hat f(y_j)+\on_j$, where
\be\label{data-pts}
y_j=(\Delta\al j_1,\Delta p j_2)\in\vs,\ \Delta p=\e,\ \Delta\al=\mu\e,\ j=(j_1,j_2)\in J,
\ee
$\on_j$ represents random noise, and $J:=\{j\in\mathbb Z^2:y_j\in\vs\}$. Likewise, $J_\al:=\{j_1\in\mathbb Z:\al_{j_1}\in\I_\al\}$ and $J_p:=\{j_2\in\mathbb Z:p_{j_2}\in\I_p\}$. For simplicity, the dependence of $J$, $J_\al$, and $J_p$ on $\e$ is suppressed.

We assume that $f$ is a sufficiently regular function supported in $\us_b \Subset \us$. Since we study noise, the only property of $f$ needed here is that $\|\hat f\|_{L^{\infty}(\vs)} \le C$ for some \textit{known} constant $C$. For the complete set of assumptions on $f$ required for image resolution analysis, see \cite[Assumption~3.3]{katsevich2025}.

Let $\Eb\on$ denote the expectation of a random variable $\on$. We also use the shorthand notation $\bar\on=\Eb\on$.

\begin{assumptions}[Properties of noise $\on_j$]\label{ass:noise}
$\hspace{1cm}$ 
\begin{enumerate}
\item $\on_j$ are independent (but not necessarily identically distributed) random variables, 
\item\label{moments} One has
\be\label{noise var}
\Eb\on_j=0,\
\Eb\on_j^2=\sigma^2(y_j)\Delta\al,\
\Eb|\on_j|^3=O(\e^{3/2}),\
\Eb\on_j^4=O(\e^{3/2}),
\ee
where $\sigma\in L^\infty(\vs)$, and all the big-$O$ terms are uniform in $j\in J$ as $\e\to0$. 
\end{enumerate}
\end{assumptions}

Prior to reconstruction, we modify the data $g(y_j)$ by applying a hard-thresholding procedure: we keep $g(y_j)$ if $|g(y_j)|\le 2C$, and set $g(y_j)=0$ if $|g(y_j)|>2C$. This is equivalent to replacing the original noise $\on_j$ with a new noise $\eta_j$, defined as follows:
\be\label{new noise}
\eta_j:=
\begin{cases} 
\on_j,&\text{if }|g(y_j)|\le 2C\\
-\hat f(y_j),&\text{if }|g(y_j)|>2C,
\end{cases}
\ j\in J.
\ee
Thus, instead of $g(y_j)$, our data are $\hat f(y_j)+\eta_j$.

The proof of the following result is in Appendix~\ref{sec:eta props}.
\begin{lemma}\label{lem:eta props} The random variables $\eta_j$, $j\in J$, are independent and they satisfy
\bs\label{eta var}
&\text{If $\on_j=0$ for some $j\in J$, then $\eta_j=0$ for that $j$},\\
&\Eb\eta_j=O(\e^{3/2}),\
\Eb\eta_j^2=\big(\sigma^2(y_j)+O(\e^{1/2})\big)\Delta\al,\\
&\Eb|\eta_j-\bar\eta_j|^3=O(\e^{3/2}),
\es
where the big-$O$ terms are uniform in $j\in J$ as $\e\to0$. 
\end{lemma}


The interpolated data, denoted $g_\e(y)$, is computed by: 
\bs\label{interp-data}
&g_\e(y):=\sum_{y_j\in\vs} \ik_\e(y-y_j)(\hat f(y_j)+\eta_j),\ y\in\vs,\\
&\ik_\e(y):=\ik(y/\e),\ \ik(y):=\ik_\al(\al/\mu)\ik_p(p),
\es
where $\ik$ is an interpolation kernel. 
\begin{assumptions}[Properties of the interpolation kernel, $\ik$]\label{ass:interp ker} 
One has
\begin{enumerate}
\item\label{ikcont} $\ik_\al$ and $\ik_p$ are compactly supported, $\ik_\al\in L^\infty(\br)$, and $\ik_p^{(3)}\in L^\infty(\br)$.
\item\label{ikeven} $\ik_*$ is even: $\ik_*(u)=\ik_*(-u)$, $u\in\br$, $*=\al,p$.
\end{enumerate}
\end{assumptions}

Given that the set $\Phi(\us_b,\I_\al)$ is bounded and the modified data $g_\e$ are bounded due to the hard thresholding procedure, it follows that $\Vert g_\e\Vert_{L^2(\vs)}\le c$, $0<\e\ll1$, with probability 1. The reconstruction is achieved by minimizing the quadratic functional
\be\label{min pb}
f_\e=\text{argmin}_{f\in H_0^1(\us_b)}\Psi(f),\ 
\Psi(f):=\Vert \R f-g_\e \Vert_{L^2(\vs)}^2+\kappa \e^{3}\Vert \pa_x f\Vert_{L^2(\us)}^2,
\ee
with some fixed regularization parameter $\kappa>0$. \cite[Lemma 3.5]{katsevich2025} yields the following result.

\begin{lemma}\label{lem:unique sol} Suppose Assumptions~\ref{ass:Phi}--
\ref{ass:interp ker} are satisfied. With probability 1, the solution to \eqref{min pb} exists and is unique for each $0<\e\ll1$.
\end{lemma}


The minimization problem in \eqref{min pb} is quadratic, hence the reconstructed image depends linearly on input data (see also \eqref{reg main eq} below). More precisely, we talk here about the modified data in \eqref{interp-data}. Therefore, in the rest of the paper we assume that $\hat f(y_j)\equiv0$, and the data consist only of noise: $g(y_j)=\eta_j$, $j\in J$. Thus,
\be\label{interp-noise}
g_\e(y)=\tsum_{j\in J} \ik_\e(y-y_j)\eta_j,\ y\in\vs.
\ee
The corresponding reconstructed image is denoted $\nrec$. 

We study $\nrec(x)$ in an $\e$-neighborhood of a point $x_0$, which is assumed to satisfy the following conditions:
\begin{assumptions}[Properties of $x_0$]\label{ass:x0} 
\hspace{1cm}
\begin{enumerate} 
\item\label{Phi prpr} The equation $\Phi_{\al\al}^{\prime\prime}(x_0,\al)=0$, $\al\in\I_\al$, has finitely many solutions. 
\item\label{sig cont} $\sigma(y)$ is Lipschitz continuous in a neighborhood of the set 
\be\label{Gamma x0}
\Gamma_{x_0}:=\{y\in\vs:x_0\in\s_y\}.
\ee
\item\label{Ial_aux} There exists an open set $I_\al\in\I_\al$ such that
\be
\sigma(\al,\Phi(x_0,\al))\not=0\  \forall \al\in I_\al.
\ee
\end{enumerate}
\end{assumptions}

\section{Main results}\label{sec:main_res}

Pick any $L\in\N$ and choose $L$ distinct points $\chx_1,\dots,\chx_L\in\br^n$. The corresponding reconstructed vector is 
\be\label{Nvec}
\vec N_\e^{\text{rec}}:=(\nrec(x_0+\e\chx_1),\dots,\nrec(x_0+\e\chx_L))\in\br^L.
\ee
For simplicity, the dependence of $\vec N_\e^{\text{rec}}$ on $x_0$ and $\chx_l$ is suppressed from notation. We remind the reader the notion of weak convergence of random variables.

\begin{definition}[{\cite[p. 185]{Khoshnevisan2002}}]
Suppose $X_n,\ 1\leq n\leq \infty$, are random variables. Then $X_n$ converges to $X_{\infty}$ weakly, which is denoted $X_n\Rightarrow X_{\infty}$, as $n\to\infty$ if the distribution of $X_n$ converges to that of $X_\infty$.
\end{definition}

\begin{theorem} \label{thm:Lyapunov_nd} Let $x_0\in\us$ and $\chx_l\in\br^n$, $l=1,2,\dots,L$, be fixed. Suppose Assumptions~\ref{ass:Phi}, \ref{geom GRT}, \ref{ass:noise}, \ref{ass:interp ker}, and \ref{ass:x0} are satisfied. Suppose $f$ is compactly supported and $\Vert\R f\Vert_{L^\infty(\vs)}\le C$ for some known $C$. The limit $\vec N^{\text{rec}}:=\lim_{\e\to0}\vec N_\e^{\text{rec}}$ exists in the sense of weak convergence and is a Gaussian random vector.
\end{theorem}

The theorem is proven in Section~\ref{sec:grid recon}. Before we state our next result, we remind the reader the definition of a random field (also known as a random function or topological space-valued random variable).

\begin{definition}\cite[p. 182]{Khoshnevisan2002}
Let $T$ be a topological space, endowed with its Borel field $\CB(T)$. A $T$-valued random variable $X$ on the probability space $(\Omega,\CG,\BP)$ is a measurable map $X:\Omega\to T$. In other words, for all $E\in\CB(T)$, $\{\om\in\Omega:X(\om)\in E\}\in\CG$.
\end{definition}

Let $D\subset R^n$ be a bounded Lipschitz domain. 
Define $C:=C(\bar D,\br)$ to be the collection of all functions continuous up to the boundary $f:\bar D\to\br$ metrized by
\be\label{Cmetric}
d(f,g)=\max_{\chx\in D}|f(\chx)-g(\chx)|,\ f,g\in C.
\ee

Recall that $N(x)$, $x\in \bar D$, is a GRF if $(N(x_1),\cdots,N(x_L))$ is a Gaussian random vector for any $L\ge1$ and any collection of points $x_1,\cdots,x_L\in \bar D$ \cite[Section 1.7]{adler2010}. As is known, a GRF is completely characterized by its mean $\bar G(x)$ and covariance function $\text{Cov}(x,y)=\Eb\big[(N(x)-\overline{N}(x))(N(y)-\overline{N}(y))\big]$, $x,y\in\bar D$ \cite[Section 1.7]{adler2010}. Thus, Theorem~\ref{thm:Lyapunov_nd} implies that $N^{\text{rec}}(\chx)$, $\chx\in \bar D$, is a zero mean GRF. 
For simplicity, we drop the dependence of $N^{\text{rec}}$ and its covariance function on $x_0$ from notation.

%

In the next theorem, we show that $\nrec(x_0+\e\check x)\Rightarrow N^{\text{rec}}(\chx)$, $\e\to0$, as $C$-valued random variables (\cite[p. 185]{Khoshnevisan2002}). 

\begin{theorem}\label{GRF_thm}
Let $D$ be a bounded domain with a Lipschitz boundary. Suppose the assumptions of Theorem~\ref{thm:Lyapunov_nd} hold. Let the function $G$ be defined as in \eqref{G def}. Then $\nrec(x_0+\e \chx)\Rightarrow N^{\text{rec}}(\chx)$, $\chx\in \bar D$, $\e\to0$, as GRFs, i.e. in the sense of weak convergence of $C$-valued random variables. Furthermore, $N^{\text{rec}}(\chx)$, $\chx\in \bar D$, is a GRF with zero mean and covariance
\bs\label{Cov main}
&\text{Cov}(\chx,\chy)
=C(\chx-\chy),\\
&C(\chx):=\int_{\I_\al}(G\star G)\big(x_0,\al,\dd_x\Phi(x_0,\al)\cdot \chx\big)\sigma^2(\al,\Phi(x_0,\al))\dd \al,\\
&(G\star G)(x_0,\al,p):=\int_\br G(x_0,\al,p+t)G(x_0,\al,t)\dd t,
\es
and sample paths of $N^{\text{rec}}(\chx)$ are continuous with probability $1$.
\end{theorem}

The theorem is proven in Section~\ref{sec:prf GRF}.

\section{Regularized solution}\label{sec:reg_sol}

\subsection{Equation for the solution to \eqref{min pb}}
Let $\R^*$ be the adjoint of $\R:L^2(\us)\to L^2(\vs)$:
\bs\label{Radj def}
(\R^* g)(x)=& \int_{\vs} g(y)W(x,y)\de(p-\Phi(x,\al)) \dd y\\
=&\int_{\I_\al} g(\al,\Phi(x,\al))W(x,(\al,\Phi(x,\al)))\dd \al,\ x\in\us.
\es
Assumption~\ref{geom GRT} ensures that $\R^*\R$ is an elliptic $\Psi$DO \cite{qu-80, Holm24}, \cite[Section VIII.6.2]{trev2} of order -1 in $\us$, i.e. $\R^*\R\in L^{-1}(\us)$. 

As is shown in \cite{katsevich2025}, the solution to \eqref{min pb} is the unique solution to the following equation 
\be\label{reg main eq}
(\R^*\R-\kappa\e^3\Delta)\nrec=\R^*g_\e \text{ in } H^{-1}(\us_b),\ \nrec\in H_0^1(\us_b).
\ee
Here we use that $(H_0^1(\us_b))^* = H^{-1}(\us_b)$; see \cite[Chapter~4, Proposition~5.1]{taylor2011}. Recall that elements of $H^s(\Omega)$ (for a domain $\Omega$) can be viewed as equivalence classes of functions in $H^s(\br^n)$, where two functions are identified if they agree on $\Omega$ \cite[Section~1.6.1]{egs}. We also use that $C^\infty(\overline{\Omega})$ is dense in $H^{-s}(\Omega)$ for $s<0$ \cite[Exercise~7, Section~4.5]{taylor2011}.

Recall that we write $\nrec$ instead of $\frec$ because $g(y_j)=\eta_j$, $j\in J$. 

\begin{lemma}\label{lem:bdd sol} Let $\psi_\e[h]$ be the solution to \eqref{min pb} (equivalently, \eqref{reg main eq}) with $g_\e$ replaced by $h\in L^2(\vs)$. One has 
\be\label{bdd sol}
\Vert \psi_\e[h] \Vert_{H^{-1/2}(\us)}\le c\Vert h\Vert_{L^2(\vs)}
\ee
for any $0<\e\ll1$ and $h\in L^2(\vs)$.
\end{lemma}

\begin{proof}
Since $\Psi(\psi\equiv0)=\Vert h \Vert_{L^2(\vs)}^2$ (cf. \eqref{min pb}), the solution $\psi_\e$ satisfies 
$\Vert \R \psi_\e\Vert_{L^2(\vs)}\le 2\Vert h \Vert_{L^2(\vs)}$, $0<\e\ll1$. By construction, $\psi_\e\in H_0^1(\us_b)\subset H_0^{-1/2}(\us_b)$, where $\us_b$ is bounded. Application of \eqref{inverse bound} completes the proof.
\end{proof}


Let $\tilde Q$ be the complete symbol of $\R^*\R$, i.e. $\R^*\R=\op(\tilde Q(x,\xi))$ (see \cite[Theorem 3.4]{Grig1994}). Let $\us_0$ be a small neighborhood of $x_0$. Pick a function $\chi\in\coi(\us)$ such that $\chi(x)\equiv1$ in $\us_0$. 

Let $\CB_\e\in L^{-2}(\us)$ be a parametrix for the operator on the left in \eqref{reg main eq}. Additionally, $\CB_\e\in L^1(\us)$, and its seminorms (i.e., the smallest constants $c_m$, $c_{m_1,m_2}$ in \eqref{symbols def}) are bounded with respect to $0<\e\ll1$. Thus,
\bs\label{soln 1}
\nrec(x)=&\big(\op\big(\tb_\e(x,\xi)\big)\chi(z)\R^* g_\e\big)(x)\\
&+\big(\op\big(\tb_\e(x,\xi)\big)(1-\chi(z))\R^* g_\e\big)(x)\\
&+\int_{\us_b} T_\e(x,x-z)\nrec(z)\dd z,\ x\in\us_0.
\es
Here $T(x,w)$ is the smooth kernel of a regularizing $\Psi$DO $\T_\e$, whose existence follows from the parametrix construction \cite[Appendix to Section I.4]{trev1}. 

\subsection{Approximation of $\nrec$}\label{ssec:approx nrec}

It is noted in \cite[(9.1) and Appendix C.2]{katsevich2025} that
\bs\label{R*R D0}
\tilde Q_0(x,\xi):=&\lim_{\e\to0}|\xi/\e|\tilde Q(x,\xi/\e)= 2\pi\sum_k\frac{|\xi|}{|\nu|} \frac{W^2(x,y)}{|\Delta_\Phi(x,\al)|},\\
y=&y_k(x,\xi),\nu=\nu_k(x,\xi),\ (x,\xi)\in\, \us\times(\br^2\setminus 0_2).
\es
Here 
\bs\label{Delta det}
\Delta_\Phi(x,y)=\text{det}\bma \dd_x\Phi(x,\al) \\ \dd_x \Phi_\al^\prime(x,\al) \ema,
\es
and the functions $y_k(x,\xi)\in \vs$ and $\nu_k(x,\xi)\in\br$ are local solutions to the equations:
\be\label{pxi v2}
p=\Phi(x,\al),\ \xi=-\nu\dd_x\Phi(x,\al),\ (\al,p)\in \vs.
\ee
The existence of at least one local solution follows from Assumption~\ref{geom GRT}\eqref{visib}. The local smoothness of solutions follows from Assumption~\ref{geom GRT}\eqref{li}. The function $\tilde Q_0(x,\xi)$ is positively homogeneous of degree zero in $\xi$. 

Denote
\be\label{tb0 def}
\tb_0(x,\hxi):=\frac{|\hxi|}{\tilde Q_0(x,\hxi)+\kappa|\hxi|^3},\ x\in \us,\hxi\in\br^2\setminus 0_2.
\ee
Denote further
\begin{align}
\label{K0 ker}
K_0(\cht;x,\al):=&\frac{1}{2\pi}\int_{\br}\tb_0(x,\la \dd_x\Phi(x,\al))e^{-i\la\cht}\dd\la,\\
\label{vth def}
\vartheta(\cht;x,\al)
=&\int_{\tsp\ik_\al} \ik_p\big(\cht+\mu\Phi_\al^\prime(x,\al)\cha\big)\ik_\al(\cha)\dd \cha,\\
\label{G def}
G(x,\al,\chq):=&\mu W(x,(\al,\Phi(x,\al)))\int_{\br}K_0(\chq-\cht;x,\al)\vartheta(\cht;x,\al)\dd\cht.
\end{align}

The following result is proven in Appendix~\ref{sec:Feps}.

\begin{lemma}\label{lem:Feps approx}
Under the assumptions of Theorem~\ref{thm:Lyapunov_nd} one has
\bs\label{Fe v4}
\nrec(x_0+\e\chx)=&\sum_{j\in J}\big[G(x_0,\al_{j_1},r_{j_1}-j_2)+R_j\big]\eta_j,\\
r_{j_1}:=&\frac{\Phi(x_0,\al_{j_1})}\e+\dd_x\Phi(x_0,\al_{j_1})\cdot\chx,\ |R_j|\le c\e\ln(1/\e).
\es
\end{lemma}

This lemma is a key ingredient in the proof of Theorem~\ref{thm:Lyapunov_nd} in Section~\ref{sec:grid recon}.

\subsection{Example}\label{ssec:CRT example}
As a brief initial illustration of Theorems~\ref{thm:Lyapunov_nd} and ~\ref{GRF_thm}, we consider the classical Radon transform in the plane, which integrates over lines $\s_y=\{x\in\br^2:\vec\al\cdot x=p\}$, $\vec\al:=(\cos\al,\sin\al)$. In this case $\I_\al=[0,2\pi)\sim S^1$, $\I_p=\br$, $\Phi(x,\al)=\vec\al\cdot x$, $W(x,y)\equiv 1$, and
\bs\label{CRT stuff}
&\dd_x\Phi(x,\al)=\vec\al,\ |\dd_x\Phi(x,\al)|=1,\ \Phi_\al^\prime(x,\al)=\vec\al^\perp\cdot x,\\ 
&\Delta_\Phi(x,y)=\text{det}\bma \dd_x\Phi(x,\al) \\ \dd_x \Phi_\al^\prime(x,\al) \ema=1,\ \vec\al^\perp:=(-\sin\al,\cos\al).
\es
From \eqref{R*R D0} and \eqref{tb0 def}:
\be
\tilde Q_0(x,\xi)=4\pi,\ \tb_0(x,\hxi)=\frac{|\hxi|}{4\pi+\kappa|\xi|^3}.
\ee
Here we use that each singularity is visible in the data two times. Finally, from \eqref{K0 ker}--\eqref{G def},
\bs
K_0(\cht;x,\al)=&\frac1\pi\ioi\frac{\la}{4\pi+\kappa\la^3}\cos(\la\cht)\dd\la,\\
\vartheta(\cht;x,\al)
=&\int_{\tsp\ik_\al} \ik_p\big(\cht+\mu(\al^\perp\cdot x)\cha\big)\ik_\al(\cha)\dd \cha,\\
G(x,\al,\chq):=&\mu\int_{\br}K_0(\chq-\cht;x,\al)\vartheta(\cht;x,\al)\dd\cht.
\es

See Section~\ref{sec:numerics} for a more detailed and extensive example involving the GRT that integrates over circles.

\section{Proof of Theorem~\ref{thm:Lyapunov_nd}}\label{sec:grid recon}

Recall that the random vector $\vec N_\e^{\text{rec}}$ is defined in \eqref{Nvec}. Pick any vector $\vec\theta\in \br^L$. By \eqref{Fe v4},
\bs\label{recon dttpr 1}
\zeta_\e:=&\vec\theta\cdot \vec N_\e^{\text{rec}} = \sum_{l=1}^L \theta_l \nrec(x_0+\e\chx_l)= \sum_{j\in J} \zeta_{\e,j},\\ 
\zeta_{\e,j}=& \bigg[\sum_{l=1}^L \theta_l \big[G\big(\al_{j_1},r_{j_1}^{(l)}-j_2\big)+R_j^{(l)}\big]\bigg]\eta_j,\\ 
r_{j_1}^{(l)}:=&\frac{\Phi(x_0,\al_{j_1})}\e+\dd_x\Phi(x_0,\al_{j_1})\cdot\chx_l.
\es 
To prove that $\vec N_\e^{\text{rec}}$ converges in distribution to a Gaussian random vector, it suffices to show that for any $\vec\theta\in \br^L\setminus0_L$, $\lim_{\e\to0}\zeta_{\e}$ is a Gaussian random variable \cite [Theorem 10.4.5]{athreya2006}. The following result is proven in Appendix~\ref{sec:variance}.

\begin{lemma}\label{lem:Lyapunov_nd} Suppose the assumptions of Theorem~\ref{thm:Lyapunov_nd} are satisfied. One has
\be\label{exp Fe}
\Eb \nrec(x_0+\e\chx)=O(\e^{1/2}\ln(1/\e)),\ \e\to0,
\ee
for any fixed $\chx$. Moreover, with $\zeta_{\e,j}$ defined in \eqref{recon dttpr 1}, one has
\be\label{main lim Lnd}
\lim_{\e\to0}\frac{\sum_{j\in J}\Eb\lvert \zeta_{\e,j}-\bar\zeta_{\e,j}\rvert^3}{\big[\sum_{j\in J}\text{Var}(\zeta_{\e,j})\big]^{3/2}}=0.
\ee
\end{lemma}

Lemma~\ref{lem:Lyapunov_nd} implies that the family of random variables, $\zeta_{\e,j}-\bar\zeta_{\e,j}$, $\e>0$, satisfies the Lyapunov condition for triangular arrays \cite [Definition 11.1.3]{athreya2006}. It follows from \cite [Corollary 11.1.4]{athreya2006}, the first line in \eqref{recon dttpr 1}, and \eqref{exp Fe} that $\zeta=\lim_{\e\to0} \zeta_{\e}$ is a zero mean Gaussian random variable. Hence, by \cite [Theorem 10.4.5]{athreya2006}, $\lim_{\e\to0}\vec N_\e^{\text{rec}}$ is a zero mean Gaussian random vector, where the limit is in the sense of weak convergence. This completes the proof of Theorem~\ref{thm:Lyapunov_nd}.

\section{Proof of Theorem \ref{GRF_thm}}\label{sec:prf GRF}
Our goal is to show that $\nrec(x_0+\e\chx)\Rightarrow N^{\text{rec}}(\chx)$, $\chx\in D$, as $C$-valued random variables as $\e\to0$. To do this, we use the following definition and theorem.

\begin{definition}[{\cite[p. 189]{Khoshnevisan2002}}]\label{def:kh}
Let $\BP_n$ be the distribution of a $C$-valued random variable $X_n$, $1\leq n\leq \infty$. The collection $(X_n)$ is tight if for all $\de\in(0,1)$, there exists a compact set $\Gamma_\de\in C$ such that $\sup_n \BP(X_n\not\in\Gamma_\de)\le \de$. 
\end{definition}

By finite-dimensional distributions of a $C$-valued random variable $X(\chx)$, $\chx\in D$, we mean the collection of all the distributions of the $\br^L$-dimensional random variables $\{X(\chx_1), X(\chx_2),\cdots,X(\chx_L)\}$ as we vary $L\geq 1$ and the points $\chx_1,\dots,\chx_L\in D$, see e.g. \cite[p. 154]{Khoshnevisan2002}.

\begin{theorem}[{\cite[Proposition 3.3.1, p. 197]{Khoshnevisan2002}}]\label{kh_thm}
Suppose $X_n,\ 1\leq n\leq \infty$, are $C$-valued random variables. Then $X_n\Rightarrow X_{\infty}$  provided that:
\begin{enumerate}
\item Finite dimensional distributions of $X_n$ converge to that of $X_{\infty}$.
\item $(X_n)$ is a tight sequence.
\end{enumerate}
\end{theorem}

The preceding definition and theorem are formulated for sequences of $C$-valued random variables indexed by $n$ and the limit is taken as $n\to\infty$. With obvious modifications, these notions and results extend to families of random variables indexed by a continuous variable $\e>0$ and when the limit is taken as $\e\to0$.

By Theorem~\ref{thm:Lyapunov_nd}, all the finite-dimensional distributions of $\nrec(x_0+\e\chx)$ converge to that of $N^{\text{rec}}(\chx)$. Thus it remains to verify Property 2 of Theorem \ref{kh_thm}. 
The following result is proven in Appendix~\ref{sec:tight}.

\begin{lemma}\label{lem:tight} $(N_{\e}^{\text{rec}}(x_0+\e\check{x}))$ is a tight family. 
\end{lemma}

By Theorem~\ref{kh_thm}, $\nrec(x_0+\e\chx)\Rightarrow N^{\text{rec}}(\chx)$ as $C$-valued random variables. Since $C$ is a complete metric space, it follows that $N^{\text{rec}}(\chx)$ has continuous sample paths with probability 1.

By Lemma~\ref{lem:Lyapunov_nd}, $N^{\text{rec}}(\chx)$ is a zero mean GRF. To completely characterize this GRF, we calculate its covariance function 
\be
\text{Cov}(\chx,\chy)=\Eb(N^{\text{rec}}(\chx)N^{\text{rec}}(\chy)),\  \chx,\chy\in D.
\ee
In fact, essentially this has already been done in the proof of Lemma~\ref{lem:Lyapunov_nd}, see Section~\ref{sec:prf outline}. From \eqref{K0 ker}--\eqref{G def} and \eqref{recon dttpr 4} we obtain
\bs\label{C_lm}
\text{Cov}&(\chx,\chy)=\mu^2\int_{\I_{\al}} W^2(x_0,\Phi(x_0,\al))\sigma^{2}(\alpha,\Phi(x_0,\al))\\
&\times\CF_{\la\to r}\big(\big|\tb_0(x_0,\la \dd_x\Phi(x_0,\al))\tilde\vartheta(\la;x_0,\al)\big|^2\big)
\dd\alpha,\ r=\al\cdot (\chx - \chy).
\es
Here $\tilde\vartheta$ is the 1D Fourier transform of $\vartheta$ with respect to the first argument.

\section{Numerical experiments}\label{sec:numerics}

The GRT integrates over circles $\s_y$, $y=(\al,\rho)$, with various radii $\rho>0$ and centers $R\vec\al$, $\al\in[0,2\pi)$, where $R=10$. The integration weight is $W\equiv 1$. Thus, $\rho=\Phi(x,\al)=|x-R\vec\al|$ and
\bs\label{needed quantities}
&\dd_x\Phi(x,\al)=\frac{x-R\vec\al}{|x-R\vec\al|}=:\vec\Theta,\
\Phi_\al^\prime(x,\al)=R\vec\Theta\cdot \vec\al^\perp,\\ 
&\dd_x\Phi_\al^\prime(x,\al)=-\frac R{|x-R\vec\al|}\vec\al^\perp+c\vec\Theta,\\ 
&\Delta_\Phi(x,\al)=\det \bma \vec\Theta \\ -\frac R{|x-R\vec\al|}\vec\al^\perp+c\vec\Theta \ema
=\frac {R\vec\al\cdot\vec\Theta}{|x-R\vec\al|}\not=0,\ |x|<R.
\es
Here $\vec\al^\perp=(-\sin\al,\cos\al)$, and $\vec\Theta^\perp$ is defined similarly. The continuous data corresponds to $\al\in\I_\al:=[0,2\pi)\sim S^1$ and $\rho\in\I_p:=(0,2R)$. The image domain is $\us=\{|x|<R\}$. The support of the object is contained in the square reconstruction region $\us_b:=(-R_{rec},R_{rec})^2$, $R_{rec}=3.7$. This corresponds to the set $\us_b$ used in \eqref{min pb}.

As is easily seen, for each $(x,\xi)\in T^*\us_b\setminus0_2$ there are two solutions, $y_k(x,\xi)$, $\nu_k(x,\xi)$, $k=1,2$, to \eqref{pxi v2}. To find $y_{1,2}$, we first solve $|x+t\vec\xi|=R$. This gives two values $t_1,t_2$, $t_2<0<t_1$. Then $\al_k$ is determined from $x+t_k\xi=R\vec\al_k$, and $\rho_k=|x-R\vec\al_k|$, $k=1,2$. By a simple geometric argument (cf. \eqref{needed quantities}), 
\be
|R\vec\al_k\cdot\vec\Theta_k|=|R\vec\al_1-R\vec\al_2|/2,\ k=1,2.
\ee

Further, substituting the above into \eqref{R*R D0} gives
\bs
\tilde Q_0(x,\xi)=&2\pi\sum_{k=1}^2 \frac{|\dd_x\Phi(x,\al_k)|}{|\Delta_\Phi(x,\al_k)|}
=2\pi\bigg[\frac {|x-R\vec\al_1|}{|R\vec\al_1\cdot\vec\Theta_1|}+\frac {|x-R\vec\al_2|}{|R\vec\al_2\cdot\vec\Theta_2|}\bigg]\\
=&2\pi\bigg[\frac {|x-R\vec\al_1|}{|R\vec\al_1-R\vec\al_2|/2}+\frac {|x-R\vec\al_2|}{|R\vec\al_1-R\vec\al_2|/2}\bigg]=4\pi.
\es
Combining with \eqref{tb0 def}--\eqref{G def} we obtain
\bs\label{G 1st}
G(x,\al,q)=&\frac{\mu}{2\pi} \int \frac{|\la|}{4\pi+\kappa|\la|^3}\tilde\ik_\rho(\la)\tilde\ik_\al(-\mu R\vec\Theta\cdot \vec\al^\perp \la) e^{-i\la q} \dd\la,
\es
where $\vec\Theta$ is the same as in \eqref{needed quantities}.

It is shown in \cite{katsevich2025} that the circular GRT used in this section satisfies Assumptions~\ref{ass:Phi} and \ref{geom GRT}. Since $\Phi(x,\al)$ is analytic in $\al$, Assumption~\ref{ass:x0}\eqref{Phi prpr} holds for every $x_0$ with $|x_0|<R$.

The discrete data are given at the points
\bs\label{d data}
&\al_{j_1}=j_1\Delta\al,\ \Delta\al=2\pi/N_\al,\ 0\le j_1<N_\al,\ N_\al=300,\\
&\rho_{j_2}=\rho_{\min}+j_2 \Delta\rho,\ \e:=\Delta\rho=(\rho_{\max}-\rho_{\min})/(N_\rho-1),\ 
0\le j_2<N_\rho,\\
&\rho_{\min}=R- R_{rec}\sqrt2,\ \rho_{\max}=R+R_{rec}\sqrt2,\ N_\rho=451.
\es

Before reconstruction, the GRT data $g(y_j)$ are interpolated from the grid \eqref{d data} to a denser grid using \eqref{interp-data}. This refined grid, where $g_\e$ is specified, has the same structure as \eqref{d data} but with $N_\al^\prime>N_\al$ and $N_\rho^\prime>N_\rho$. For this interpolation, we use the Keys interpolation kernel \cite{Keys1981, btu2003}
\bs\label{keys}
\ik_\al(t)=&\ik(t/\mu),\ \ik_\rho(t)=\ik(t),\ \mu=\Delta\al/\Delta\rho, \\
\ik(t)=&3B_3(t) - (B_2(t-1/2) + B_2(t + 1/2)),
\es
where $B_n$ is the cardinal $B$-spline of degree $n$ supported on $[-(n+1)/2, (n+1)/2]$. Therefore $\tsp\ik=[-2,2]$. The kernel is a piecewise-cubic polynomial with continuous $\ik,\ik^\prime$ and bounded $\ik^{\prime\prime}$, so $\ik\in C_0^2(\br)$. As is well-known, 
\be\label{FT spline}
(\CF B_n)(\la)=\text{sinc}(\la/2)^{n+1},\ \text{sinc}(t):=\sin t/t.
\ee
Therefore,
\be\label{FT keys}
\tilde\ik(\la)=(\text{sinc}\,t)^3[3\,\text{sinc}\,t-2\cos t],\ t=\la/2.
\ee

Substituting \eqref{keys} and \eqref{FT keys} into \eqref{G 1st} we find:
\bs\label{G 2nd}
G(x,\al,q)=&\frac{\mu}{\pi} \ioi \frac{\la}{4\pi+\kappa\la^3}\tilde\ik(\la)\tilde\ik(-\mu a(x,\al)\la) \cos(\la q) \dd\la,
\es
where $a(x,\al)$ is the distance from the origin to the chord of the circle $|x|=R$ passing through the points $x$ and $R\vec\al$. Similarly, the covariance function is computed by (cf. \eqref{Cov main}, \eqref{C_lm})
\bs\label{covar 1st}
C(\check x)=\frac{\mu^2}{\pi}\int_0^{2\pi} \ioi & \bigg[\frac{\la}{4\pi+\kappa\la^3}\tilde\ik(\la)\tilde\ik(-\mu a(x_0,\al)\la)\bigg]^2 \cos(\la \vec\Theta\cdot\chx) \dd\la\\ 
&\times\sigma^2(\al,|x_0-R\vec\al|)\dd\al,\ \vec\Theta=\frac{x_0-R\vec\al}{|x_0-R\vec\al|}.
\es

In order to approximate the continuous GRT $\R$ and its adjoint $\R^*$, the reconstruction is performed on a dense $801\times801$ grid covering the same square region $\us_b$. The regularization parameter is set at $\kappa=0.5$. 

The functional \eqref{min pb} is minimized using gradient descent. The stopping criterion is met when either
\be
\text{(i)}\ \Vert f_{k+1}-f_k\Vert_{L^\infty(\us_b)}<10^{-4} \text{ or } \text{(ii)}\ \Psi(f_k)\le 10^{-6}\Psi(f_0) 
\ee
occurs three times (not necessarily consecutively). Here $f_k$ denotes the image at the $k$-th iteration, and $f_0(x)\equiv0$. Occurrences of (i) and (ii) are counted separately; mixed occurrences are not combined toward the total of three.

From the proof of Lemma~\ref{lem:Feps approx} it follows that \eqref{Fe v4} holds for any choice of $\us_b\Subset\us$, provided $x_0\in\us_b$. However, computing $\nrec(x)$ by numerically solving \eqref{min pb} is sensitive to discretization errors. To mitigate this, we chose a smaller domain $\us_b^\prime=(1,1.4)\times(0.5,0.9)$ centered at the fixed point $x_0=(1.2,0.7)$. We set $\eta_j=1$ for a single index $j_0=(j_1,j_2)=(100,229)$ and $\eta_j=0$ for all $j\neq j_0$. The index $j_0$ was selected so that $\rho_{j_2}\approx |x_0-R\vec\al_{j_1}|$, with $\al_{j_1}=2.1$ and $\rho_{j_2}=10.1$. The data are then interpolated to a fine grid with $N_\al^\prime=2400$ and $N_\rho^\prime=3601$. Finally, the reconstruction uses an $801\times 801$ grid covering $\us_b^\prime$ (rather than the original $\us_b$). The large values of $N_\al^\prime$ and $N_\rho^\prime$, together with a dense reconstruction grid on a small region, ensure high accuracy.

\begin{figure}[h]
{\centerline{
{\epsfig{file={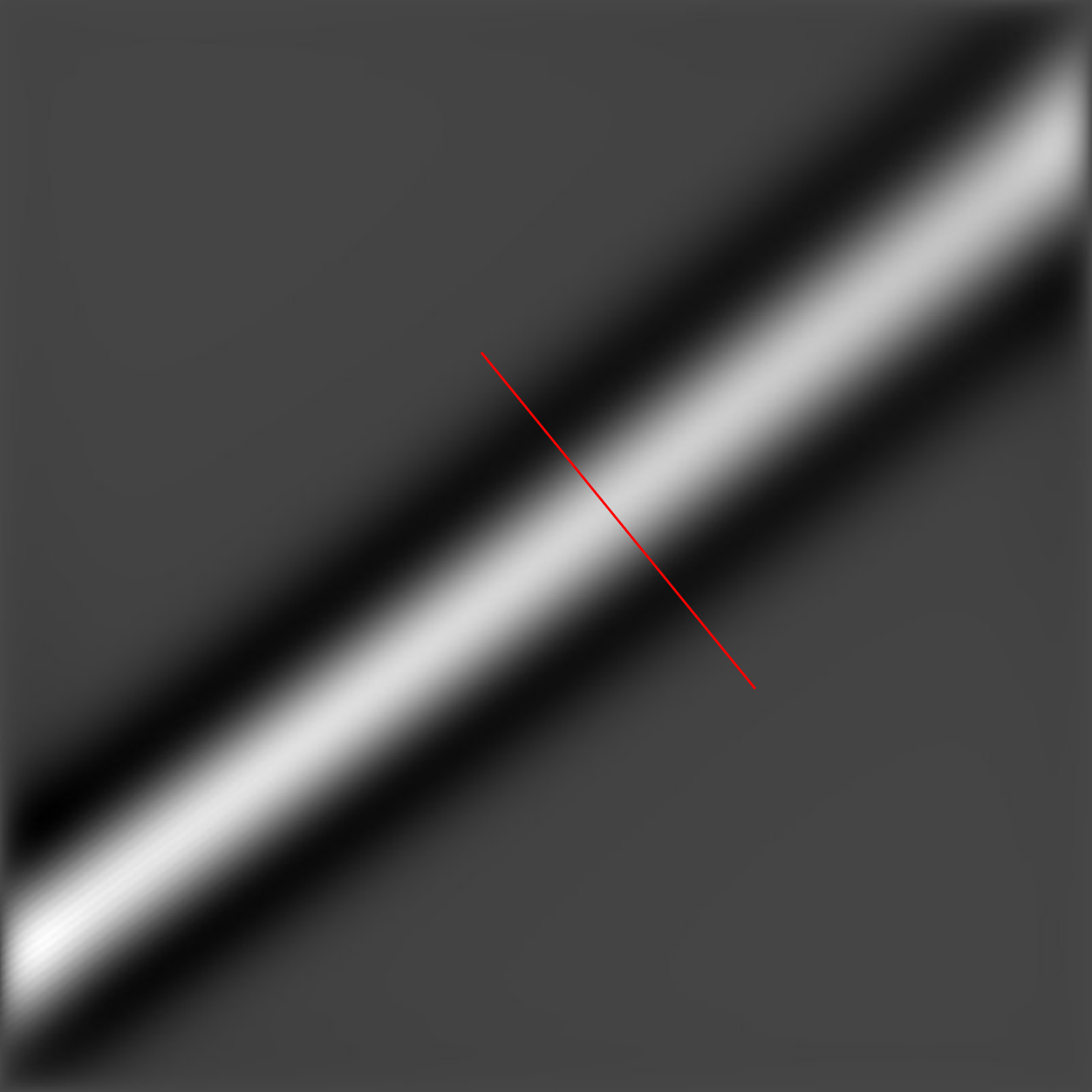}, width=4.5cm}}
{\epsfig{file={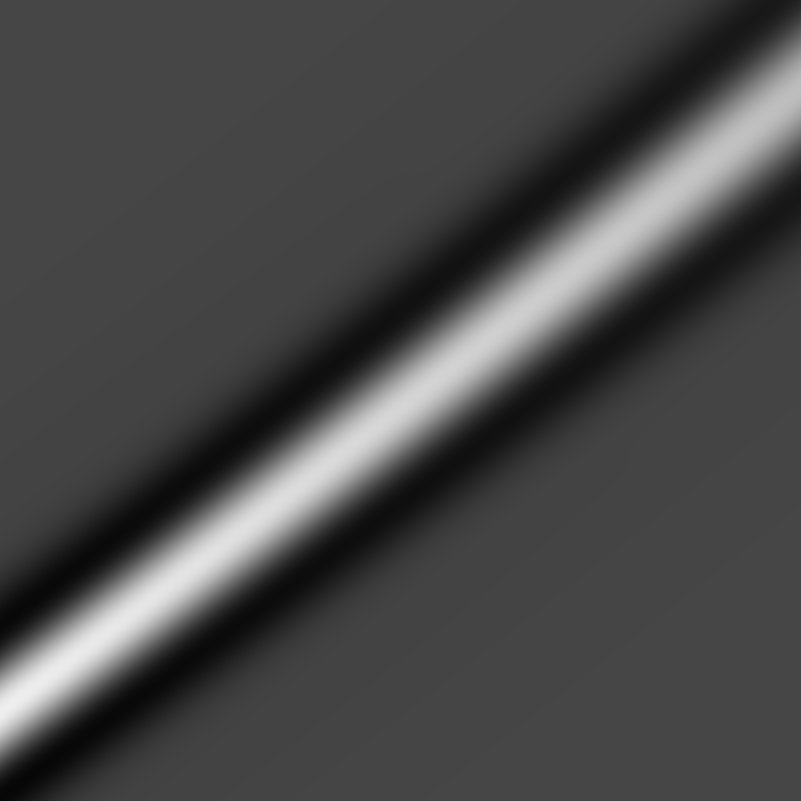}, width=4.5cm}}
}}
\caption{Left: iteratively reconstructed $\nrec(x)$ in the square $\us_b^\prime = [1,1.4]\times[0.5,0.9]$. Right: $G(x_0,\al_{j_1},(|x - R\vec{\al}_{j_1}| - \rho_{j_2})/\e)$ for $x \in \us_b^\prime$, computed using \eqref{G 2nd}.}
\label{fig:Gfn_loc}
\end{figure}

The results are shown in Figures~\ref{fig:Gfn_loc}, \ref{fig:Gfn_diff}, and \ref{fig:line_plots}. The left panel of Figure~\ref{fig:Gfn_loc} shows the density plot of the iteratively reconstructed $\nrec(x)$ for $x\in\us_b^\prime$, while the right panel displays the function $G(x_0,\al_{j_1},(|x-R\vec\al_{j_1}|-\rho_{j_2})/\e)$, $x\in\us_b^\prime$, computed using \eqref{G 2nd}, over the same domain. Both panels use a display window of $[-0.02,0.06]$. Figure~\ref{fig:Gfn_diff} shows the difference between these two images, with a tighter display window of $[-5\cdot10^{-3},5\cdot10^{-3}]$. Consistent with the lemma, the difference is negligible except near the boundary of $\us_b^\prime$. To better illustrate the agreement, Figure~\ref{fig:line_plots} plots cross-sections of the three images along the red line shown in the left panel of Figure~\ref{fig:Gfn_loc}. 

\begin{figure}[h]
{\centerline{
{\epsfig{file={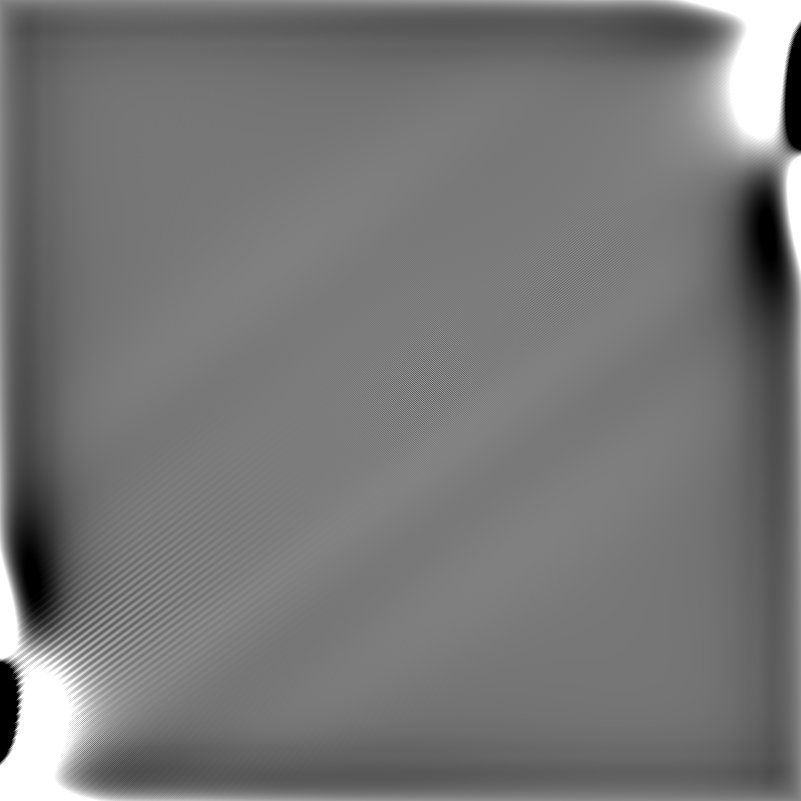}, width=4.5cm}}
}}
\caption{Difference between the two images shown in Figure~\ref{fig:Gfn_loc}, displayed with the window $[-5\cdot10^{-3},5\cdot10^{-3}]$.}
\label{fig:Gfn_diff}
\end{figure}

\begin{figure}[h]
{\centerline{
{\epsfig{file={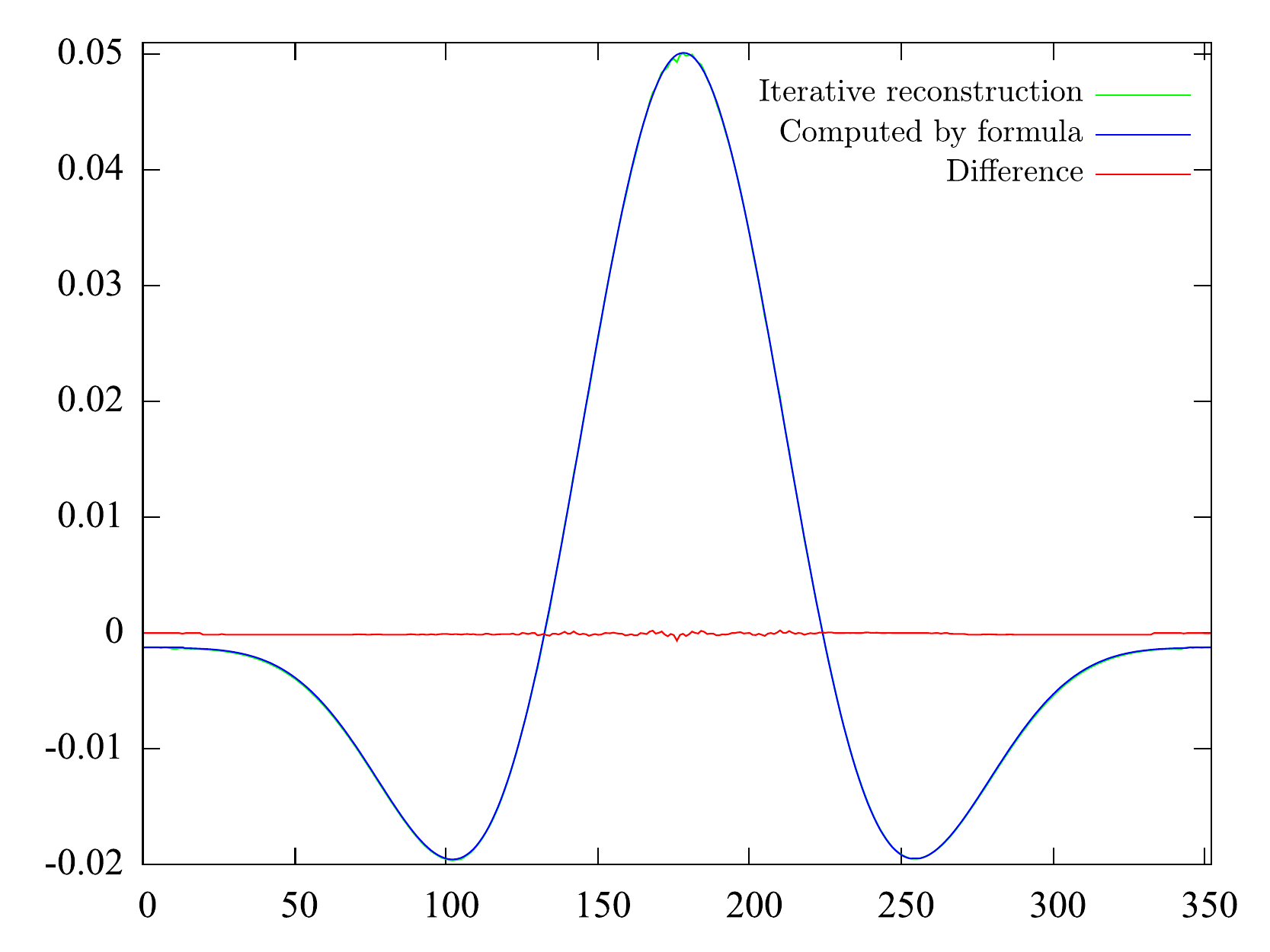}, width=10cm}}
}}
\caption{Cross-sections of the three images from Figures~\ref{fig:Gfn_loc} and \ref{fig:Gfn_diff} along the red line indicated in the left panel of Figure~\ref{fig:Gfn_loc}.}
\label{fig:line_plots}
\end{figure}

To validate the predicted covariance function $C(\chx)$ in \eqref{Cov main}, we simulate a complete data set with noisy measurements of a test function $f$. We choose $f$ to be the characteristic function of the disk centered at $x_c=(1,1)$ with radius $r=2$.  The center of a local region of interest (ROI) is the same point $x_0=(1.2,0.7)$ as before. Then we simulate 1500 noisy data sets. For each data set, the noise values $\on_j$ are drawn from the uniform distribution on $[-1,1]$, which has a standard deviation of $1/3$, and then scaled by $(3\Delta\al)^{1/2}\sigma(y_j)$, where
\be
\sigma(y)=(1 + 0.5\sin(2\al))(1 - 0.4\cos\rho)\chi_{\tsp(\R f)}(y).
\ee
Here $\chi_{\tsp(\R f)}(y)$ is the characteristic function of the support of $\R f$. This choice ensures that Assumptions~\ref{ass:noise} and \ref{ass:x0}(\ref{sig cont},\ref{Ial_aux}) are satisfied. 

In the spirit of conventional X-ray CT, if the curve $\s_{y_j}$ does not intersect the disk, the corresponding noise value $\on_j$ is set to zero.  Physically, in X-ray CT a beam passing only through air experiences negligible attenuation, so its associated noise level is much smaller than that of beams traversing the object and undergoing substantial attenuation. The discontinuity of $\sigma(y)$ at the boundary of the projected disk in the data domain is not problematic, since this discontinuity remains at a positive distance from $\Gamma_{x_0}$ (cf. Assumption~\ref{ass:x0}\eqref{sig cont}). 

To reduce computation time, the data interpolation grid was reduced to $N_\al^\prime=1200$, $N_\rho^\prime=1801$. The entire original $\us_b=[-R_{rec},R_{rec}]^2$ was reconstructed on an $801\times 801$ grid. We recorded the reconstructed values at the center $x_0$ and at a nearby point $x_1=x_0+\e\chx$, where $\chx=(1.24,-1.77)$. The sample variances at these points are 0.042 and 0.039, respectively, compared to the predicted value of 0.043. Figure~\ref{fig:histograms} shows the sample and predicted histograms, which agree well. The sample covariance between $\nrec(x_0)$ and $\nrec(x_1)$ is $0.024$, while the predicted value, computed using \eqref{Cov main}, is $C(\chx)=0.027$.

\begin{figure}[h]
{\centerline{
{\epsfig{file={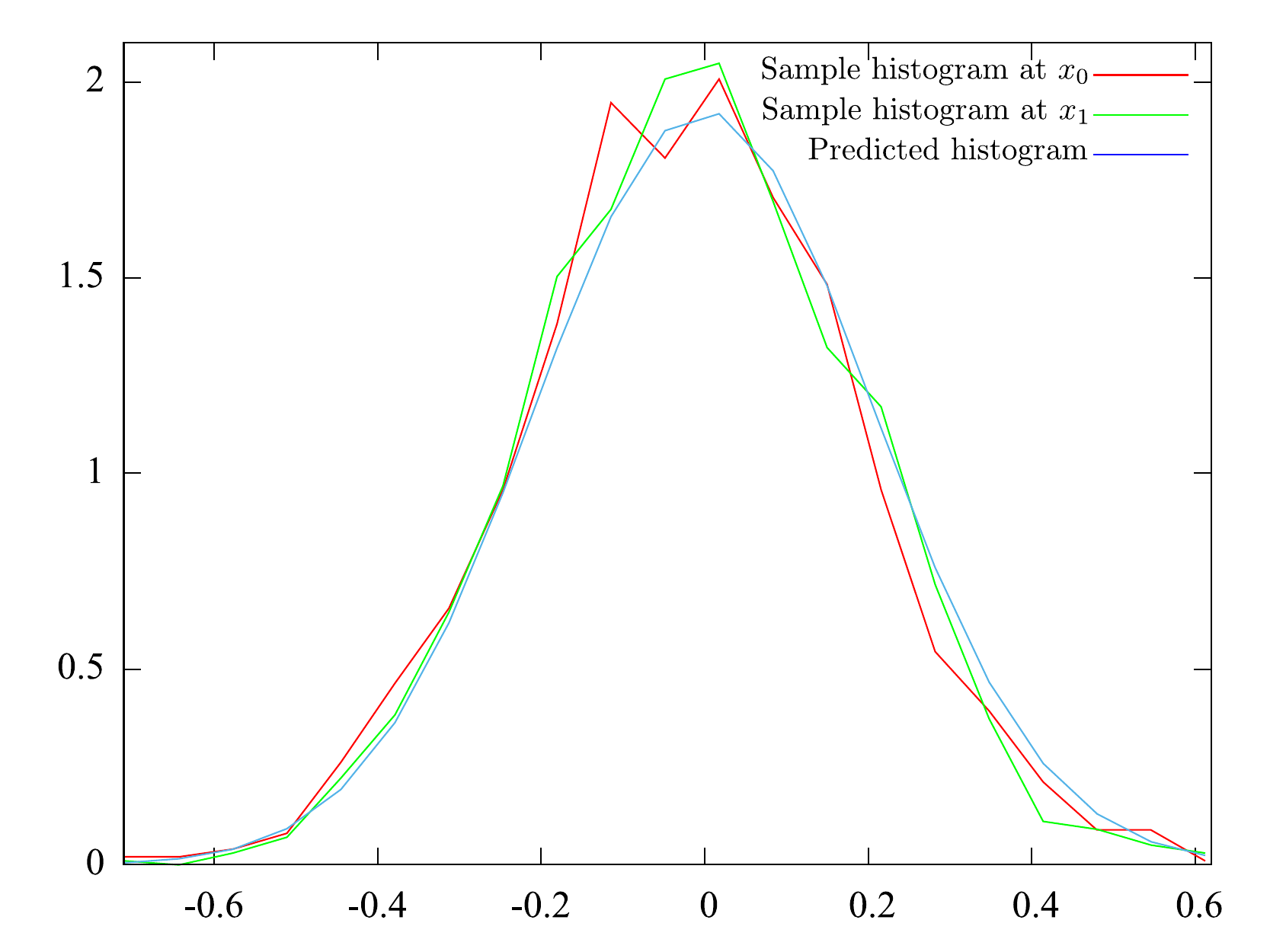}, width=10cm}}
}}
\caption{Sample histograms of the computed values $\nrec(x_0)$ (red) and $\nrec(x_1)$ (green), along with the predicted histogram (blue), which does not depend on the point.}
\label{fig:histograms}
\end{figure}

\appendix

\section{Proof of Lemma~\ref{lem:eta props}}\label{sec:eta props}

Since each $\eta_j$ is a function of only one of the independent variables $\on_j$, the variables $\eta_j$ are also independent. For simplicity, we drop the subscript $j$ for the remainder of the proof and denote $\hat f:=\hat f(y_j)$. 

If $\on=0$, then $|\hat f|\le C$ implies that the corresponding $g$ is not replaced with zero. In this case, $\eta=\on=0$.

Let $\pi$ be the probability measure of $\on$. By construction,
\bs\label{exp diff}
\Eb|\eta-\on|&\le \int_{|x+\hat f|\ge 2C}|x+\hat f|\dd \pi(x)
\le 2\int_{|x|\ge C} |x|\dd \pi(x)\\
&\le \frac2{C^2}\int_{|x|\ge C} |x|^3\dd \pi(x)\le  \frac2{C^2}\Eb|\on|^3=O(\e^{3/2}).
\es
Using the triangle inequality and that $\Eb\on=0$ proves the claim about $\Eb\eta$. Arguing similarly we obtain
\bs\label{exp diff 2}
\Eb|\eta^2-\on^2|&\le \int_{|x+\hat f|\ge 2C}x^2\dd \pi(x)
\le \int_{|x|\ge C} x^2\dd \pi(x)\\
&\le \frac1{C}\int_{|x|\ge C} |x|^3\dd \pi(x)=O(\e^{3/2}).
\es
In the first inequality we used that $|\on^2-\eta^2|=|x^2-(-f)^2|\le x^2$ on the domain of integration. Combining with \eqref{noise var} proves the claim about $\Eb\eta_j^2$.


Finally,
\be\label{diff3 expanded}
|\eta-\bar\eta|^3\le c \big(|\eta-\on|^3+|\on-\bar\eta|^3\big)\le c\big(|\eta-\on|^3+|\on|^3+|\bar\eta|^3\big)
\ee
Similarly to \eqref{exp diff},
\bs\label{exp diff 3}
\Eb|\eta-\on|^3&\le \int_{|x+\hat f|\ge 2C}|x+\hat f|^3\dd \pi(x)
\le 8\int_{|x|\ge C} |x|^3\dd \pi(x)\\
&\le \frac8C\int_{|x|\ge C} x^4\dd \pi(x)\le  \frac8C\Eb\on^4=O(\e^{3/2}).
\es
Taking the expectation on both sides of \eqref{diff3 expanded} and using the previously obtained results and \eqref{noise var} we prove the last claim of the lemma
\be\label{exp diff3 res}
\Eb|\eta-\bar\eta|^3=O(\e^{3/2}).
\ee

\section{Proof of Lemma~\ref{lem:Feps approx}}\label{sec:Feps}

In view of the first line in \eqref{soln 1}, we introduce the function
\bs\label{get lead term v1}
F_\e(x)&:=\big(\op(\tb_\e(x,\xi))\chi(z)\R^*g_\e\big)(x).
\es
Throughout this section, we assume that $x\in\us_0$.

\subsection{Replacing $\tb_\e$ with a simpler amplitude in \eqref{get lead term v1}}\label{ssec:glob anal}
In view of \eqref{get lead term v1}, \eqref{reg main eq}, \eqref{R*R D0}, introduce the symbols
\be\label{B0pr}
\tb_0^\prime(x,\xi):=\frac{|\xi|}{\tilde Q_0(x,\xi)+\kappa\e^3|\xi|^3},\
\Delta\tb_\e(x,\xi):=\tb_\e(x,\xi)-\tb_0^\prime(x,\xi). 
\ee
As is easily seen from \eqref{R*R D0} and the parametrix construction \cite[Appendix to Section I.4]{trev1}, $\Delta\tb_\e\in S^0(\us)$ with seminorms uniformly bounded for $0<\e\ll1$.

Using \eqref{interp-data}, \eqref{interp-noise}, \eqref{Radj def} and \eqref{get lead term v1}, define
\bs\label{glt 1}
J_\e(x;y_j)&:=\frac{\e\mu}{(2\pi)^2}\int_{\br^2}\int_{\us}\tb_\e(x,\xi)\chi(z)e^{i\xi\cdot(z-x)}\\
&\hspace{2cm}\times\int_{\br} W(z,(\al,p))
\ik_p\bigg(\frac{p-p_{j_2}}\e\bigg)\ik_\al(\cha)\dd \cha\dd z\dd\xi,\\ p&=\Phi(z,\al),\ \al=\al_{j_1}+\e\mu\cha.
\es
The factor $\e\mu$ appears due to the change of variables $\al\to\cha=(\al-\al_{j_1})/(\e\mu)$. 
Similarly, define
\bs\label{glt 2}
\Delta J_\e(x;y_j)&:=\frac{\e\mu}{(2\pi)^2}\int_{\br^2}\int_{\us}\Delta\tb_\e(x,\xi)\chi(z)e^{i\xi\cdot(z-x)}\\
&\hspace{1cm}\times\int_{\br} W(z,(\al,p))
\ik_p\bigg(\frac{p-p_{j_2}}\e\bigg)\ik_\al(\cha)\dd \cha\dd z\dd\xi.
\es

\begin{lemma}\label{lem:del J bnd}
Under the assumptions of Theorem~\ref{thm:Lyapunov_nd}, one has $|\Delta J_\e(x;y_j)|\le c\e$, $x\in\us_0$, $y_j\in\vs$,\ $0<\e\ll1$.
\end{lemma}

\begin{proof}
Our goal is to obtain a bound for $\Delta J_\e(x;y_j)$ as $\e\to0$. To this end we first omit the integral with respect to $\cha$ in \eqref{glt 2} and drop the prefactor:
\bs\label{glt 3}
\Delta H_\e(x;\cha,y_j):=&\int_{\us}\int_{\br^2}\Delta\tb_\e(x,\xi) e^{i\xi\cdot(z-x)}\dd\xi\\ 
&\hspace{1cm}\times W_1(z,\cha)\ik_p\bigg(\frac{\Phi(z,\al)-p_{j_2}}\e\bigg)\dd z,\\
W_1(z,\cha):=&\chi(z) W(z,(\al,p)).
\es
The integral with respect to $\xi$ is not absolutely convergent and understood in the sense of distributions. In \eqref{glt 2} and \eqref{glt 3}, $p$ and $\al$ are the same as in \eqref{glt 1}. For convenience, in the rest of the proof we drop the dependence of various quantities on $\cha$ and $y_j$. By construction, $\al-\al_{j_1}=O(\e)$ in \eqref{glt 3}. In simplified notation, \eqref{glt 3} becomes
\be
\Delta H_\e(x)=\int_{\us}\int_{\br^2}\Delta\tb_\e(x,\xi) e^{i\xi\cdot(z-x)}\dd\xi\,
W_1(z)\ik_p\bigg(\frac{\Phi(z)-p_{j_2}}\e\bigg)\dd z.
\ee

Pick any point $z_0\in\s_{y_j}$, i.e. so that $\Phi(z_0,\al_{j_1})=p_{j_2}$. Using that $\ik_p$ is compactly supported and appealing to a partition of unity argument, we can assume without loss of generality that $\tsp W_1\subset\us_1$, where $\us_1$ is a small neighborhood of $z_0$. 

Pick any $x\in\us_0$ such that $\text{dist}(x,\us_1)\ge c_0>0$ for a sufficiently small $c_0$. Therefore, $|x-z|$, $z\in\us_1$, is bounded away from zero. Integrate by parts with respect to $\xi$ multiple times and use that $\Delta\tb_\e\in S^0(\us)$ to obtain an integral with respect to $\xi$ that converges absolutely and uniformly in $x$. This implies that $|\Delta H_\e(x)|\le c$ for all $0<\e\ll1$. The constant $c$ can be selected independently of $x\in\us_0$ provided that $\text{dist}(x,\us_1)\ge c_0$. 

Suppose now that $\text{dist}(x,\us_1)\le c_0$. By decreasing $c_0>0$ and shrinking $\us_1$, if necessary, we can assume that $x$ is sufficiently close to $z_0$. Using that $\dd_x\Phi(z)\not=0$, $z\in\us_1$, we can change variables $z=\psi(w)$ so that $\Phi(z)-p_{j_2}\equiv w_1$ and $\det(\pa z/\pa w)$ is bounded. Using \cite[Theorem 3.3, Chapter I]{trev1} and \cite[Theorem 1.5]{egorov1994}, we can transform the expression for $\Delta H_\e$ to the form
\be
\Delta H_\e(x)=\int_{\Xi}\int_{\br^2}\Delta\tb_\e^{(1)}(v,w,\eta) e^{i\eta\cdot(w-v)}\dd\eta\,
W_1(\psi(w))\ik_p(w_1/\e)\dd w
\ee
with some $\Delta\tb_\e^{(1)}\in S^0(\Xi)$, where $\Xi=\psi^{-1}(\us_1)$ and $v=\psi^{-1}(x)$. The absolute value of the Jacobian $|\det(\pa z/\pa w)|$ is included in $\Delta\tb_\e^{(1)}$. Changing variables gives
\bs\label{delHe}
\Delta H_\e(x)=\int_{\br^2}\bigg[\int_{\br^2}&\Delta\tb_\e^{(1)}\big(v,w,(\heta_1/\e,\eta_2)\big) W_1\big(\psi(\e\chw_1,w_2)\big)\\
&\times e^{i\eta_2(w_2-v_2)}\dd w_2\,\ik_p(\chw_1)e^{i\heta_1(\chv_1-\chw_1)}\dd \chw_1\bigg]\dd\heta_1\dd\eta_2.
\es
Using that $|\Delta\tb_\e^{(1)}(v,w,\eta)|\le c$, $v,w\in\Xi$, $\eta\in\br^2$, and that $W_1$, $\ik_p$ are sufficiently smooth and compactly supported, we can show that the integral in brackets decays sufficiently fast as $\heta_1,\eta_2\to\infty$. 

Specifically, the symbol $\Delta\tb_\e^{(1)}\big(v,w,(\heta_1/\e,\eta_2)\big)$ and $W_1$ are smooth in $\chw_1$ and $w_2$. Hence, the expression in brackets is absolutely integrable with respect to $\eta_2$ uniformly with respect to $x\in \us_0$ and $\heta_1\in\br$. To prove decay, it therefore suffices to require that $\tilde\ik_p\in L^1(\br)$, a property that  follows directly from Assumption~\ref{ass:interp ker}\eqref{ikcont}. Indeed, since $\ik_p^\prime\in L^\infty(\br)$ and $\ik_p$ is compactly supported, it follows that $\tilde\ik_p\in L^1(\br)$.

For convenience, the domain of integration with respect to $(\chw_1,w_2)$ is extended to all of $\br^2$. This does not change the value of the integral due to the assumption $\tsp W_1\subset\us_1$. The absolute convergence of the integral with respect to $(\heta_1,\eta_2)$ implies that $|\Delta H_\e(x)|\le c$ for all $x\in\us_0$ close to $z_0$ (if they exist) and $0<\e\ll1$. 

The above argument proves that $|\Delta H_\e(x)|\le c$ for all $x\in\us_0$ and $0<\e\ll1$. Integrating with respect to $\cha$ in \eqref{glt 2} and taking into account the factor $\e$ proves the lemma.
\end{proof}

\subsection{Simplification of the leading order term}\label{ssec: simpl lot}

In view of Lemma~\ref{lem:del J bnd}, replace $\tb_\e$ with $\tb_0^\prime$ in \eqref{glt 1} and change variables to obtain
\bs\label{glt 4}
J_\e(x;y_j):=&\frac{\e^{-2}\mu}{(2\pi)^2}\int_{\us}\int_{\br^2}\tb_0(x,\hxi)\chi(z)e^{i\hxi\cdot(z-x)/\e}\dd\hxi\\
&\hspace{0cm}\times\int_{\br} W(z,(\al,p))
\ik_p\bigg(\frac{\Phi(z,\al)-p_{j_2}}\e\bigg)\ik_\al(\cha)\dd \cha\dd z,\\  
p=&\Phi(z,\al),\ \al=\al_{j_1}+\e\mu\cha,
\es
where we have redefined the meaning of $J_\e$. Recall that $\tb_0$ is defined in \eqref{tb0 def}. As is easily checked, $\e\tb_0^\prime(x,\hxi/\e)=\tb_0(x,\hxi)$.

Introduce the kernel
\be\label{B0 expr}
B_0(x,\chw):=\frac{1}{(2\pi)^2}\int_{\br^2}\tb_0(x,\hxi)e^{-i\hxi\cdot\chw}\dd\hxi.
\ee
Later we will use the following result:
\be\label{B0 bnds}
B_0(x,\chw)=\begin{cases} O(-\ln|\chw|),& |\chw|\to0,\\ 
O(|\chw|^{-3}),& |\chw|\to\infty.
\end{cases}
\ee
The top case follows from \cite[Theorem 5.12]{abels12} because $\tb_0\in S^{-2}(\us)$. The bottom case is proven completely analogously to \cite[Lemma~8.1]{katsevich2025}. 

Using $B_0$, represent $J_\e$ in the form
\bs\label{Je v2}
J_\e(x;y_j)=&\e^{-2}\mu\int_{\us}B_0\bigg(x,\frac{x-z}\e\bigg)\chi(z)\\
&\qquad\times\int_{\br} W(z,(\al,p))
\ik_p\bigg(\frac{\Phi(z,\al)-p_{j_2}}\e\bigg)\ik_\al(\cha)\dd \cha\dd z,\\
p=&\Phi(z,\al),\ \al=\al_{j_1}+\e\mu\cha.
\es
Change variables $z\to\chw=(x-z)/\e$:
\bs\label{Je v2b}
J_\e(x;y_j)=&\mu\int_{\e^{-1}(x-\us)}B_0(x,\chw)\chi(x-\e\chw)\\
&\times\int_{\br} W(x-\e\chw,(\al,p))
\ik_p\bigg(\frac{\Phi(x-\e\chw,\al)-p_{j_2}}\e\bigg)\ik_\al(\cha)\dd \cha\dd\chw,\\
p=&\Phi(x-\e\chw,\al),\ \al=\al_{j_1}+\e\mu\cha.
\es
Using \eqref{B0 bnds} and that the domain of integration with respect to $\chw$ is contained in a ball $|\chw|\le c/\e$ gives
\bs\label{Je v2c}
J_\e(x;y_j)=&\mu W\big(x,(\al_{j_1},\Phi(x,\al_{j_1}))\big)\chi(x)\\
\times\int_{\br^2}&B_0(x,\chw) 
\vartheta\bigg(\frac{\Phi(x,\al_{j_1})-p_{j_2}}\e-\dd_x\Phi(x,\al_{j_1})\chw\bigg)\dd\chw+O(\e\ln(1/\e)).
\es

To obtain \eqref{Je v2c} from \eqref{Je v2b}, we first incorporate the function $\vartheta$, introduced in \eqref{vth def}, by using that 
\bs\label{aux fracs}
&\frac{\Phi(x-\e\chw,\al)-p_{j_2}}\e\\
&=\frac{\Phi(x-\e\chw,\al_{j_1})-p_{j_2}}\e
+\mu\Phi_\al^\prime(x-\e\chw,\al_{j_1})\cha+O(\e)\\
&=\bigg[\frac{\Phi(x,\al_{j_1})-p_{j_2}}\e-\dd_x\Phi(x,\al_{j_1})\chw+O(\e(1+|\chw|))\bigg]
+\mu\Phi_\al^\prime(x,\al_{j_1})\cha.
\es
Ignoring the big-$O$ term in the last line in \eqref{aux fracs} leads to an overall error of $O(\e\ln(1/\e))$ in \eqref{Je v2c}. This follows from the properties $\ik_p^\prime\in L^\infty(\br)$ (see Assumption~\ref{ass:interp ker}\eqref{ikcont}), $B_0(x,\chw)=O(|\chw|^{-3})$, $|\chw|\to\infty$, and the fact that the domain of integration in \eqref{Je v2b} is contained in a ball $|\chw|\le c/\e$.

Then we extend the domain of integration from $\e^{-1}(x-\us)$ to $\br^2$, which contributes a term of magnitude $O(\e)$. This follows from the fact that the set $x-\us$ is a neighborhood of zero, $\vartheta$ is compactly supported in its first argument, $|\dd_x\Phi(x,\al_{j_1})|\ge c$, and $B_0(x,\chw)=O(|\chw|^{-3})$, $|\chw|\to\infty$. To simplify notation, the arguments $x,\al_{j_1}$ of $\vartheta$ are omitted.

Next, suppose $x=x_0+\e\chx$. It is easy to show that the partial derivatives $\pa_{x_j} B_0(x,\chw)$, $j=1,2$, also satisfy \eqref{B0 bnds}. This can be established using the same method of proof as for $B_0$. Using also that $\chi(x)\equiv1$ on $\us_0$ we obtain
\bs\label{Je v2d}
J_\e(x_0+\e\chx;y_j)=&\mu W\big(x_0,(\al_{j_1},\Phi(x_0,\al_{j_1}))\big)\\
\times\int_{\br^2}B_0(x_0,\chw) &
\vartheta\bigg(\frac{\Phi(x_0+\e\chx,\al_{j_1})-p_{j_2}}\e-\dd_x\Phi(x_0,\al_{j_1})\chw\bigg)\dd\chw\\
&+O(\e\ln(1/\e)).
\es
The integral with respect to $\chw$ can be simplified as follows
\bs\label{chw int}
&\int_{\br^2}B_0(x_0,\chw) 
\vartheta\big(\chq-\dd_x\Phi(x_0,\al_{j_1})\chw\big)\dd\chw\\
&=\int K_0\big(\cht;x_0,\al_{j_1}\big)\vartheta(\chq-\cht)\dd \cht,\ \chq=\frac{\Phi(x_0+\e\chx,\al_{j_1})-p_{j_2}}\e,
\es
where
\bs\label{B int ds}
K_0(\cht;x_0,\al):=\int_{\br} B_0\big(x_0,\chw_1 e_1+\chw_2 e_2\big)\dd \chw_2.
\es
Here $e_1=|\dd_x\Phi(x_0,\al)|^{-1}\dd_x\Phi(x_0,\al)$, $e_2$ is a unit vector orthogonal to $e_1$, and $\chw=(\chw_1,\chw_2)$. Using \eqref{B0 expr} shows that $K_0$ satisfies \eqref{K0 ker}.

Clearly,
\be
\frac{\Phi(x_0+\e\chx,\al_{j_1})-p_{j_2}}\e=\frac{\Phi(x_0,\al_{j_1})-p_{j_2}}\e+\dd_x \Phi(x_0,\al_{j_1})\chx+O(\e).
\ee
We therefore approximate $\chq$ in \eqref{chw int} and \eqref{Je v2d} by the right-hand side of the last equation with the term $O(\e)$ neglected. Arguing as above and using that $\vartheta_{\cht}^\prime\in L^\infty(\br)$ in its first argument, we conclude that this approximation introduces an error of magnitude $O(\e)$ in \eqref{Je v2d}.

\subsection{End of proof}\label{ssec:end of prf} Consider now the second and third terms on the right in \eqref{soln 1}. 

The operator $\op\big(\tb_\e(x,\xi)\big)(1-\chi(z)$, $x\in\us_0$, is smoothing. Arguing analogously to the proof of Lemma~\ref{lem:del J bnd} we establish that 
\be
\big(\op\big(\tb_\e(x,\xi)\big)(1-\chi(z))\R^*\ik_\e(\cdot-y_j)=O(\e),\ x\in \us_0,\ 0<\e\ll 1.
\ee

Clearly, $\Vert \ik_\e(\cdot-y_j)\Vert_{L^2(\vs)}=O(\e)$.
Set $h(y)=\ik_\e(y-y_j)$ in \eqref{bdd sol}, and let $\psi_{\e,y_j}\in H_0^1(\us_b)$ be the corresponding solution to \eqref{reg main eq}. The inequality \eqref{bdd sol} implies $\Vert\psi_{\e,y_j}\Vert_{H^{-1/2}(\us)}\le c\e$. By the linearity of \eqref{reg main eq}, $\nrec(x)=\sum_{j}\psi_{\e,y_j}(x)\eta_j$. Applying $\T_\e$ on both sides gives
\be
(\T_\e\nrec)(x)=\tsum_{j\in J} A_{\e,y_j}(x)\eta_j,\ x\in\us_0,
\ee
for some $A_{\e,y}(x)\in C^\infty(\us_0)$ which satisfy $|A_{\e,y}(x)|\le c\e$ uniformly in $x\in K\subset\us_0$, where $K$ is any compact set, $y\in\vs$, and $0<\e\ll1$. This follows because $\T_\e:H_{comp}^{-1/2}(\us)\to H_{loc}^s(\us_0)$ is bounded for any $s\in\br$. Selecting any $s>1$ and using the bound $\Vert f\Vert_{L^\infty(\br^2)}\le c \Vert f\Vert_{H^s(\br^2)}$ \cite[Chapter 4, Proposition 1.3]{taylor2011} gives the desired result. 

Combining all the results, and using \eqref{interp-noise} along with the fact that $p_{j_2}=\e j_2$, we finish the proof.

\section{Proof of Lemma~\ref{lem:Lyapunov_nd}}\label{sec:variance}

We begin with a simple lemma, which is needed in the proof.
\begin{lemma}\label{lem:K0 estims}
One has
\be\label{KeK0bnds}
|K_0(\chq;x,\al)|\le c(1+|\chq|)^{-2},\ \chq\in\br,
\ee
and 
\be\label{K0 der bnd}
\bigg|\pa_{\chq}\int K_0(\chq-\cht;x,\al)\vartheta(\cht;x,\al)\dd \cht\bigg|\le c(1+|\chq|)^{-3},
\ee
for all $x\in\us$ and $\al\in\I_{\al}$.
\end{lemma}

The bound~\eqref{KeK0bnds} follows from \eqref{K0 ker} by integration by parts and the fact that $\tilde B_0(x,\hxi)=O(|\hxi|)$ as $\hxi\to0$, see \eqref{tb0 def}. The integral is absolutely convergent and $K_0$ is bounded because $\tilde B_0(x,\la\dd_x\Phi(x,\al))=O(\la^{-2})$ as $\la\to\infty$.

The bound~\eqref{K0 der bnd} likewise follows from integration by parts. The decay at infinity is faster due to the extra factor $\la$ in the integrand, which arises in the Fourier domain from the derivative $\pa_{\chq}$. The derivative is bounded due to the convolution with $\vartheta$, which satisfies $\vartheta_{\cht}^\prime\in L^\infty(\br)$ and is compactly supported in the first argument.

In this section, we omit $x_0$ from the arguments of $G$ for simplicity. By \eqref{KeK0bnds} and the fact that $\vartheta$ is compactly supported in $\cht$, we have
\be\label{G bnd}
|G(\al,\chq)|\le c(1+|\chq|)^{-2},\ \al\in\I_\al,\ \chq\in\br.
\ee

The assertion \eqref{exp Fe} follows from \eqref{eta var}, \eqref{Fe v4}, and \eqref{G bnd}.

\subsection{Proof of \eqref{main lim Lnd} in the case $L=1$}
Begin by setting $L=1$ and $\theta_l=1$ in \eqref{recon dttpr 1}. Denote:
\be\label{De def}
D_\e:=\tsum_{j\in J}\text{Var}(\zeta_{\e,j}),
\ee
i.e. $D_\e$ is the expression in brackets in the denominator in \eqref{main lim Lnd}. By \eqref{eta var},
\bs\label{stp 1}
&\sum_{j_2\in J_p}\big|(G(\al,r-j_2)+R_j)^2-G^2(\al,r-j_2)\big|\Eb \eta_j^2\\
&\le c\e\sum_{j_2\in J_p}\big[R_j^2+|G(\al,r-j_2)R_j|\big].
\es
Using \eqref{G bnd} and \eqref{Fe v4},  straightforward calculations show that 
\bs\label{sum R bnd}
\sum_{j_2\in J_p}R_j^2\le c\e\ln^2(1/\e),\ 
\sum_{j_2\in J_p}|G(\al,r-j_2)R_j|\le c\e\ln(1/\e).
\es
Also, by \eqref{eta var}, \eqref{Fe v4}, and \eqref{recon dttpr 1},
\bs\label{zeta exp}
\Eb\zeta_{\e,j}=&\big[G\big(\al_{j_1},r_{j_1}-j_2\big)+O(\e\ln(1/\e))\big]O(\e^{3/2}).
\es
Therefore,
\bs\label{zeta exp bnd 1}
&\tsum_{j\in J} (\Eb\zeta_{\e,j})^2\\
&\le c\e^3\sum_{j\in J} \big[G^2\big(\al_{j_1},r_{j_1}-j_2\big)+\e\ln(1/\e)\big|G\big(\al_{j_1},r_{j_1}-j_2\big)\big|+\e^2\ln^2(1/\e)\big]\\
&\le c\e^3\sum_{j_1\in J_\al} \big[1+\e\ln(1/\e)+\e\ln^2(1/\e)\big]=O(\e^2).
\es
Hence, combining \eqref{stp 1}--\eqref{zeta exp bnd 1} yields
\bs\label{Deps v1}
D_\e=&\tsum_{j\in J}\Eb\zeta_{\e,j}^2-\tsum_{j\in J}(\Eb\zeta_{\e,j})^2\\
=&\Delta\al\sum_{j_1\in J_\al}\bigg[\sum_{j_2\in J_p}G^2(\al_{j_1},r_{j_1}-j_2)\sigma^2(\al_{j_1},p_{j_2})+O(\e\ln^2(1/\e))\bigg]\\
&+O(\e^2)\\
=&\Delta\al\sum_{j_1\in J_\al}\sum_{\substack{j_2\in J_p\\|r_{j_1}-j_2|\le \e^{-1/3}}} G^2(\al_{j_1},r_{j_1}-j_2)\sigma^2(\al_{j_1},p_{j_2})+O(\e\ln^2(1/\e)).
\es
Here we have used that $\sigma$ is bounded and, by \eqref{G bnd}, limiting the sum with respect to $j_2$ to the set $|r_{j_1}-j_2|\le \e^{-1/3}$ leads to an overall error of magnitude $O(\e)$ in\eqref{Deps v1}.


Define 
\bs\label{psi_a_b}
\psi(\al,r):=&\sum_{j_2\in\BZ}  G^2(\al,r-j_2),\ \al\in\I_\al,r\in\br.
\es 
By \eqref{G bnd} the series is absolutely convergent. It is easy to see that
\begin{align}\label{Psi per}
\psi(\al,r+m)=\psi(\al,r),\  \al\in\I_\al,r\in\br,m\in\BZ.
\end{align}
Replace $p_{j_2}$ with $\Phi(x_0,\al_{j_1})$ in the arguments of $\sigma^2$ in \eqref{Deps v1} and extend the sum with respect to $j_2$ to all of $\BZ$ to obtain
\be\label{Deps}
D_\e=\Delta\al\sum_{j_1\in J_\al} \psi(\al_{j_1},r_{j_1})\sigma^2(\al_{j_1},\Phi(x_0,\al_{j_1}))+O(\e\ln^2(1/\e)).
\ee
Indeed, consider the sum with respect to $j_2$. By the definition of $r_{j_1}$ in \eqref{Fe v4} we have
$p_{j_2}-\Phi(x_0,\al_{j_1})=\e(j_2-r_{j_1}+O(1))$. By \eqref{G bnd} and the Lipschitz continuity of $\sigma$ in a neighborhood of the set $\Gamma_{x_0}$ (Assumption~\ref{ass:x0}\eqref{sig cont}):
\bs\label{sigma mod}
&\sum_{\substack{j_2\in J_p\\|r-j_2|\le \e^{-1/3}}} G^2(\al,r-j_2)|\sigma^2(\al,\e j_2)-\sigma^2(y,\Phi(x_0,\al_{j_1}))|\\
&\le c\sum_{|r-j_2|\le \e^{-1/3}} \e(1+|r-j_2|)(1+|r-j_2|)^{-4}=O(\e).
\es
Here we use the fact that all $y_j$ satisfying $|r_{j_1}-j_2|\le \e^{-1/3}$ lie in a neighborhood of $\Gamma_{x_0}$ for all $0<\e\ll1$.

Again, by \eqref{G bnd}, we have 
\be
\sum_{j_2\in\BZ\setminus J_p} G^2(\al,r-j_2)\le c\sum_{|r-j_2|\ge 1/\e} G^2(\al,r-j_2)=O(\e^{-3}),
\ee
where $r=r_{j_1}$. In the first inequality, we used that Assumption~\ref{ass:Phi}\eqref{Phi sm} implies the existence of $\de>0$ such that, for each $(x_0,\al_{j_1})\in\us_0\times\I_\al$, all $p_{j_2}$ satisfying $|\Phi(x_0,\al_{j_1})-p_{j_2}|\le\de$ are in the data, i.e. they belong to $\I_p$. 

Summing over $\al_{j_1}\in\I_\al$ and accounting for the factor $\Delta\al$ in front of the sum in \eqref{Deps v1} proves \eqref{Deps}.

Using Assumption~\ref{ass:x0} and arguments very similar to \cite[Appendix D]{katsevich2024a}, we obtain the following result
\begin{align}\label{lim De}
\lim_{\e \to 0}D_\e&=\int_{\I_\al}\bigg(\int_0^1 \psi(\al,r) \dd r\bigg)\sigma^2(\al,\Phi(x_0,\al)) \dd \al.
\end{align} 
Indeed, the argument in \cite[Appendix D]{katsevich2024a} is based on Assumption 2.8 in that paper, which consists of three items. Our Assumption~\ref{ass:x0}\eqref{Phi prpr} ensures that Assumption 2.8(1) holds, and Assumption~\ref{ass:x0}\eqref{Ial_aux} guarantees Assumption 2.8(2). Finally, Assumption 2.8(3) requires that the equation $\dd_x\Phi(x_0,\al)\cdot\chx=0$ for $\al\in\I_\al$ has a solution set of Lebesgue measure zero for each $\chx\in\br^2\setminus 0_2$. This follows from Assumption~\ref{geom GRT}\eqref{li}. Indeed, if $\dd_x\Phi(x_0,\al)\cdot\chx=0$ for some $\al\in\I_\al$, then $\pa_\al(\dd_x\Phi(x_0,\al)\cdot\chx)\neq 0$, since otherwise the vectors $\dd_x\Phi(x_0,\al)$ and $\dd_x\Phi_\al^\prime(x_0,\al)$ would both be orthogonal to $\chx$ and hence linearly dependent. This contradiction shows that solutions to $\dd_x\Phi(x_0,\al)\cdot\chx=0$ are isolated points.

Furthermore, 
\bs\label{lim psi int}
\int_0^1 \psi(\al,r)\dd r&=\sum_{j_2\in\BZ} \int_0^1G^2(\al,r-j_2)\dd r
=\int_\br G^2(\al,r)\dd r=:H(\al).
\es
By Parseval's theorem and \eqref{K0 ker}, \eqref{G def},
\bs\label{G2int}
&\int_{\br}G^2(\al,r)\dd r \\
&=\frac{\mu^2W^2(x_0,\Phi(x_0,\al))}{2\pi}\int_{\br}|\tb_0(x_0,\la \dd_x\Phi(x_0,\al))\tilde\vartheta(\la;x_0,\al)|^2\dd\la.
\es
Combined with \eqref{lim De}--\eqref{G2int}, Assumption~\ref{ass:x0}(2) yields
\begin{align}\label{A.6}
\lim_{\e \to 0}D_\e=\int_{\I_\al} H(x_0,\al)\sigma^2(\al,\Phi(x_0,\al)) \dd \al>0.
\end{align}

Consider now the numerator in \eqref{main lim Lnd} (with $L=1$). It is bounded by 
\begin{align}\label{A.7}
c\sum_{j\in J}\lvert G(\al_{j_1},r_{j_1}-j_2)+R_j\rvert^3\Eb|\eta_j-\bar\eta_j|^3.
\end{align}
Using the last result in \eqref{eta var} and the estimates for $G$ and $R_j$ we get that the numerator is $O(\e^{1/2})$. Together with \eqref{A.6} this proves Lemma~\ref{lem:Lyapunov_nd} in the case $L=1$.

\subsection{Proof of \eqref{main lim Lnd} in the case $L\ge 1$}\label{sec:prf outline}
The proof of \eqref{main lim Lnd} in the general case $L\ge 1$ is similar to the above, so we only highlight key new points. Instead of \eqref{Deps}, we now have
\bs\label{recon dttpr 2}
D_\e=&\Delta\al\sum_{j\in J} \biggl[\sum_{l=1}^L \theta_l \big(G\big(\al_{j_1},r_{j_1}^{(l)}-j_2\big)+R_j^{(l)}\big)\biggr]^2\sigma^2(\al_{j_1},\Phi(x_0,\al_{j_1}))\\
&+O(\e\ln^2(1/\e)).
\es
See \eqref{recon dttpr 1} for the definitions of $r_{j_1}^{(l)}$ and $R_j^{(l)}$. Here we have used the same substitution $p_{j_2}\approxeq \Phi(x_0,\al_{j_1})$ (cf. \eqref{sigma mod}) and \eqref{eta var}. Estimating the contribution of the remainder $R_j^{(l)}$ similarly to \eqref{sum R bnd} and passing to the limit as in \eqref{lim De} gives
\bs\label{recon dttpr 3}
\lim_{\e\to0}D_\e=&\int_{\I_\al}\int_{\br} f^2(\al,r)\dd r\, \sigma^2(\al,\Phi(x_0,\al))\dd \al,\\
f(\al,r):=&\sum_{l=1}^L  \theta_l G\big(\al,r+\dd_x\Phi(x_0,\al)\cdot\chx_l\big).
\es 
Multiplying out the sum in the definition of $f$ yields 
\bs\label{recon dttpr 4}
\lim_{\e\to0}D_\e=&\sum_{l_1=1}^L\sum_{l_2=1}^L  \theta_{l_1}\theta_{l_2}C(\chx_{l_1}-\chx_{l_2}),\\
C(\chw):=&\int_{\I_\al}(G\star G)\big(\al,\dd_x\Phi(x_0,\al)\cdot \chw\big)\sigma^2(\al,\Phi(x_0,\al))\dd\al.
\es 
Observe that if $L=1$, $\theta_1=1$, and $w=0_2$, then \eqref{recon dttpr 3} coincides with \eqref{lim De}, \eqref{lim psi int} after a change of variables $r+\dd_x\Phi(x_0,\al)\cdot \chx_1\to r$. By \eqref{G def},
\bs\label{f alt}
f(\al,r)=&\frac{\mu W(x_0,\Phi(x_0,\al))}{2\pi}\int_{\br} \tb_0(x_0,\la \dd_x\Phi(x_0,\al))\tilde\vartheta(\la;x_0,\al)\\ &\times\bigg[\sum_{l=1}^L  \theta_l e^{-i\la \dd_x\Phi(x_0,\al)\cdot\chx_l}\bigg]e^{-i\la r}\dd \la.
\es 

Suppose the limit in \eqref{recon dttpr 3} equals zero. By Assumption~\ref{ass:x0}(2), this implies that $f(\al,r)\equiv0$ for all $r\in\br$ and $\al\in I_\al$. The set $I_\al$ is introduced in Assumption~\ref{ass:x0}(2). Furthermore, $\vartheta$ is compactly supported in $\cht$, so $\tilde\vartheta$ is analytic in $\la$ and cannot be zero on an open set. Additionally, $W(x_0,\Phi(x_0,\al))>0$ and $\tb_0(x_0,\la \dd_x\Phi(x_0,\al))\not=0$ (see Assumption~\ref{geom GRT}\eqref{pos W} and \eqref{tb0 def}). Therefore
\bs\label{zero expon}
\sum_{l=1}^L  \theta_l e^{-i\la\dd_x\Phi(x_0,\al)\cdot\chx_l}\equiv 0,\ \la\in\br,\ \al\in I_\al.
\es 
Arguing very similarly to \cite[Appendix E]{katsevich2024a}, we conclude that all $\theta_l$ are zero. For convenience, this argument is presented in Section~\ref{ssec:theta zero}.

Finally, the argument to show that the numerator in \eqref{main lim Lnd} is $O(\e^{1/2})$ is very similar to \eqref{A.7}.

\subsection{Proof of the fact that $\theta_l=0$, $1\le l\le L$}\label{ssec:theta zero}
Consider the vectors $e_k:=\chx_{l_1}-\chx_{l_2}\in\br^2$, $1\le l_1,l_2\le L$, $l_1\not=l_2$, and the sets $I_{\al,k}:=\{\al\in I_\al: \dd_x\Phi(x_0,\al)\cdot e_k=0\}$. The $\chx_l$ are distinct, so $e_k\not=0$ for each $k$. By Assumption~\ref{geom GRT}\eqref{li}, if $\al$ satisfies $\dd_x\Phi(x_0,\al)\cdot e_k=0$, then $\pa_\al(\dd_x\Phi(x_0,\al)\cdot e_k)\not=0$. Hence each set $I_{\al,k}$ is a collection of isolated points. There are finitely many $I_{\al,k}$, so their union is not all of $I_\al$. Therefore there exists $\al_0\in I_\al\setminus(\cup_k I_{\al,k})$ such that $\dd_x\Phi(x_0,\al_0)\cdot e_k\not=0$ for each $k$. 

We found $\al_0\in I_\al$ such that $s_l:=\dd_x\Phi(x_0,\al_0)\cdot\chx_l$, $1\le l\le L$, are pairwise distinct. Differentiate \eqref{zero expon} with respect to $\la$ multiple times and set $\la=0$. This gives
\be\label{equations}
\sum_{l=1}^L  \theta_l s_l^k=0,\ k=0,1,\dots,L-1.
\ee 
Since all $s_l$ are pairwise distinct, using the properties of the Vandermond determinant we conclude from \eqref{equations} that all $\theta_l=0$. This contradiction proves that the limit in \eqref{recon dttpr 3} is not zero unless $\vec\theta=0$. 


\section{Proof of Lemma~\ref{lem:tight}}\label{sec:tight}

To prove the lemma we use results from the theory of Sobolev spaces. Consider the sets 
\be\label{compact sets}
\Gamma_\de:=\{f\in C:\, \Vert f\Vert_{W^{2,2}(D)}^2 \le 1/\de\},
\ee
where $D$ is a bounded Lipschitz domain. We use the Sobolev space $W^{k,p}(D)$, which is the closure of $C^\infty(\bar D)$ in the norm:
\be\label{Wkp norm}
\Vert f\Vert_{k,p}:=\bigg(\int_D\sum_{|m|\le k}|\pa_x f(x)|^p\dd x\bigg)^{1/p},\ m\in\N_0^2,\ f\in C^\infty(\bar D).
\ee

By \cite[eq. (7.8), p. 146 and Theorem 7.26, p. 171]{GilbTrud}, the imbedding $W^{2,2}(D)\hookrightarrow C(\bar D)$ is compact. More precisely, we use here that the imbedding $W^{2,2}(D)\hookrightarrow W^{2,p}(D)$, $1\le p\le 2$ is continuous \cite[eq. (7.8), p. 146]{GilbTrud}, and the imbedding $W^{2,p}(D)\hookrightarrow C(\bar D)$, $1\le p< 2$, is compact \cite[Theorem 7.26, p. 171]{GilbTrud}. Hence the set $\Gamma_\de\subset C$ is compact for every $\de>0$. 

The proof parallels that of Lemma~\ref{lem:Feps approx}. Begin with $F_\e$. Similarly to \eqref{glt 2},
\bs\label{glt 2 v2}
\Delta J_\e(x_0+\e\chx;y_j)&:=\frac{\e\mu}{(2\pi)^2}\int_{\br^2}\int_{\us}\Delta\tb_\e(x_0+\e\chx,\xi)\chi(z)e^{-i\xi\cdot w}\\
&\hspace{1cm}\times\int_{\br} W(z,(\al,p))
\ik_p\bigg(\frac{p-p_{j_2}}\e\bigg)\ik_\al(\cha)\dd \cha\dd w\dd\xi,\\
p=&\Phi(z,\al),\ \al=\al_{j_1}+\e\mu\cha,\ z=x_0+\e\chx-w.
\es
Arguing as in Appendix~\ref{ssec:glob anal}, we prove that $\pa_{\chx}^m\Delta J_\e(x_0+\e\chx;y_j)=O(\e)$, $|m|\le 2$. In particular, differentiating $\ik_p((p-p_{j_2})/\e)$ with respect to $\chx$ does not introduce factors that grow with $\e$, since the $\e$ in the numerator and denominator cancel out.

Based on this argument, the function $\ik_p^{\prime\prime}$ must satisfy the smoothness conditions imposed on $\ik_p$ in Appendix~\ref{sec:Feps}. The most restrictive of these is the requirement $\ik_p^{\prime}\in L^\infty(\br)$ (see, e.g., the paragraph following \eqref{Je v2c}), which in turn implies that $\ik_p^{(3)}\in L^\infty(\br)$. The analog of the other condition from Appendix~\ref{sec:Feps}, namely $\tilde\ik_p\in L^1(\br)$ (see the paragraph following \eqref{delHe}), 
now becomes $\tilde\ik_p^{\prime\prime}\in L^1(\br)$. As before, this requirement is less restrictive and is automatically satisfied because $\ik_p$ is compactly supported and $\ik_p^{(3)}\in L^\infty(\br)$.

Substituting $x=x_0+\e\chx$ into \eqref{Je v2b} and differentiating $J_\e(x_0+\e\chx;y_j)$ with respect to $\chx$, we conclude that the analysis in the remainder of Appendix~\ref{ssec: simpl lot} applies equally to the derivatives. Note that this differentiation does not introduce boundary terms, despite the domain of integration depending on $x$, because $\chi\in\coi(\us)$. Therefore, we can represent $\pa_{\chx}^m J_\e(x_0+\e\chx;y_j)$ in a form analogous to \eqref{Fe v4}. The analog of $G$ has the same decay at infinity, i.e. $O((r_{j_1}-j_2)^{-2})$. This gives
\be
|\pa_{\chx}^mJ_\e(x_0+\e\chx;y_j)|\le c \big[(1+|r_{j_1}-j_2|)^{-2}+\e\ln(1/\e)\big],\ |m|\le 2.
\ee
Recall that $r_{j_1}$ are defined in \eqref{Fe v4}. 

The final contribution to $\nrec$ comes from the difference $\nrec-F_\e$, see \eqref{soln 1} and \eqref{get lead term v1}. The magnitude of its derivatives with respect to $\chx$ can be estimated similarly to section~\ref{ssec:end of prf}.

Combining the results proves that
\bs
\pa_{\chx}^m\nrec(x_0+\e\chx)=&\sum_{j\in J} A_{m,j}(\chx)\eta_j,\\ 
|A_{m,j}(\chx)|\le & c \big[(1+|r_{j_1}-j_2|)^{-2}+\e\ln(1/\e)\big],\ |m|\le 2.
\es
Therefore (cf. \eqref{Deps v1})
\bs\label{var der}
\Eb(\pa_{\chx}^m N_\e^{\text{rec}}(x_0+\e\chx))^2=&c \e\sum_{j\in J} A_{m,j}^2\sigma^2(y_j)
\le c,\  \chx\in D,\ |m|\le2.
\es
As opposed to \eqref{Deps v1}, here we do not need an asymptotic expression but a simple upper bound.
This implies that $\Eb \Vert N_\e^{\text{rec}}(x_0+\e\chx))\Vert_{W^{2,2}(D)}^2\le c_0$ for some $c_0$ and all $0<\e\ll1$.

Let $\pi_\e(x)$ be the probability distribution of the non-negative random variable $\xi_\e=\Vert N_\e^{\text{rec}}(x_0+\e\chx)\Vert_{W^{2,2}(D)}^2$. By an elementary inequality,
\bs\label{cheb}
\BP(N_\e^{\text{rec}}(x_0+\e\chx)\not\in\Gamma_\de)=&\BP\big(\Vert N_\e^{\text{rec}}(x_0+\e\chx)\Vert_{W^{2,2}(D)}^2 \ge 1/\de\big)\\
=&\int_{1/\de}^\infty\dd\pi_\e(x)\le\de\int_{1/\de}^\infty x\dd\pi_\e(x)\le c_0\de.
\es
Therefore $(N_{\e}^{\text{rec}}(x_0+\e\check{x}))$ is a tight family. 

\section*{Acknowledgment}
The author acknowledges using ChatGPT for minor language editing (spelling, grammar, and general style).

\bibliographystyle{siam}
\bibliography{My_Collection}
\end{document}